\documentclass[11pt]{article}

\usepackage{color}
\usepackage[left=1in, right=1in, top=1in, bottom=1in]{geometry}
\usepackage{graphicx}
\usepackage{amsmath}
\usepackage{amsfonts}
\usepackage{amssymb}
\usepackage{url}
\usepackage{chemarr}
\usepackage{amsthm}
\usepackage{url}
\usepackage{enumerate}
\usepackage[colorlinks]{hyperref}

\theoremstyle{theorem}
\newtheorem{lemma}{Lemma}[section]

\newtheorem{theorem}{Theorem}[section]
\newtheorem{conjecture}{Conjecture}

\newtheorem{prop}{Proposition}[section]

\theoremstyle{definition}
\newtheorem{definition}{Definition}[section]
\newtheorem{example}{Example}[section]
\newtheorem{remark}{Remark}[section]

\def\edited#1{{\color{black}{#1}}}

\def\S{\mathcal{S}}
\def\Re{\mathcal{R}}
\def\C{\mathcal{C}}
\def\Z{\mathbb{Z}}
\newcommand{\EE}{\mathbb{E}}
\newcommand{\R}{\mathbb{R}}

\begin{document}

\title{Time-dependent product-form Poisson distributions for  reaction networks with higher order complexes}
\author{David F.~Anderson\thanks{University of Wisconsin-Madison, anderson@math.wisc.edu}  \and David Schnoerr\thanks{Imperial College, d.schnoerr@imperial.ac.uk} \and Chaojie Yuan\thanks{University of Wisconsin-Madison, cyuan25@math.wisc.edu}}

\maketitle

\begin{abstract}
	It is well known that stochastically modeled reaction networks that are complex balanced admit a stationary distribution that is a product of Poisson distributions.  In this paper, we consider the following related question: supposing that the initial distribution of a stochastically modeled reaction network is a product of Poissons, under what conditions will the distribution remain a product of Poissons for all time?  By drawing inspiration from Crispin Gardiner's ``Poisson representation'' for the solution to the chemical master equation, we provide a necessary and sufficient condition for  such a product-form distribution to hold for all time. Interestingly, the condition is a dynamical ``complex-balancing'' for only those complexes that have multiplicity greater than or equal to two (i.e.~the higher order complexes that yield non-linear terms to the dynamics).  We term this new condition the ``dynamical and restricted complex balance'' condition (DR for short).
\end{abstract}

\section{Introduction}

Reaction  networks are commonly utilized in the modeling of  biological processes such as gene regulatory networks, signaling networks,  viral infections, cellular metabolism, etc., and their dynamics are typically modeled in one of
three ways \cite{AK2015, schnoerr2017approximation}. If the counts of the constituent molecules are low, then the dynamics of the abundances is typically modeled stochastically with  a discrete-space, continuous-time Markov chain in $\Z^d_{\ge 0}$, where $d$ is the number of  species in the system.   \edited{If the counts are moderate then the concentrations of the constituent species are often approximated by some form of continuous diffusion process}.  However, If the counts of the constituent species are high, then the evolution of their concentrations is often modeled deterministically via a system of ordinary differential equations.

Analytic treatments of such models are rarely possible, and most existing approaches analyze steady states: fixed points of the concentrations in the deterministic modeling regime and stationary distributions in the stochastic regime. However, most biological processes are not in steady state and experiments typically measure transient dynamics. To identify the underlying interactions, time-dependent solutions of the relevant dynamical equations are needed \cite{munsky2018distribution, wilkinson2006stochastic}.  For stochastic systems modeled as discrete-space, continuous-time Markov chains, \edited{explicit formulas for the time dependent distributions of the process are rarely known except in some specific cases \cite{peccoud1995markovian,ramos2011exact}. To the best of our knowledge, the only general treatment of the time dependent behavior is derived for systems where all the reactant and product complexes (the vertices of the associated graph)  are of order zero or one  \cite{jahnke2007solving}.  
 } Because of this, either stochastic simulations or approximation methods are typically employed in the stochastic setting \cite{neuert2013systematic, schnoerr2017approximation, smadbeck2012efficient, zechner2012moment,shahrezaei2008analytical,cao2018linear}.  \edited{However, these approaches are typically computationally expensive, give rise to uncontrolled estimation errors, and/or are applicable to specific model classes \cite{schnoerr2017approximation}.}  To the best of our knowledge, the results presented in this article are the first that provide the exact time dependent distribution for a general class of reaction networks with higher order complexes.

In the series of papers \cite{F1,H,H-J1}, Feinberg, Horn, and Jackson introduced the  notion of network deficiency and proved that if the reaction network  (i) is weakly reversibility and (ii) has a deficiency of zero, then the resulting deterministically modeled system endowed with mass action kinetics is ``complex balanced,'' regardless of the choice of rate parameters.    See \cite{AndProdForm,AndKurtz2011,AK2015} for  terminology.
In \cite{AndProdForm}, Anderson, Craciun, and Kurtz proved a corresponding result for the associated jump Markov models.  In particular, they showed that any stochastic model whose deterministic counterpart is complex balanced (which, by the works cited above, includes all models whose network is weakly reversible and has a deficiency of zero) admits a  stationary distribution that is a product of Poissons. The specific distribution is
\begin{equation}\label{eq:789876}
	\pi(x) = \prod_{i=1}^d e^{-\tilde c_i} \frac{\tilde c_i^{x_i}}{x_i!}, \quad x\in \Z_{\ge 0}^d,
\end{equation}
where $\tilde c \in \R^d_{>0}$ is a  complex balanced fixed point of the corresponding deterministically modeled system.   See also \cite{ACKK2018}, where the processes considered in \cite{AndProdForm} were shown to be non-explosive, and \cite{AC2016}, where the main result from \cite{AndProdForm} was generalized to a class of models with non-mass action kinetics.  Finally, the interested reader may also see \cite{CW2016}, where a converse to the main theorem in \cite{AndProdForm} is shown.  Specifically, they show that if a system admits the stationary distribution \eqref{eq:789876}, then the associated deterministic model is complex balanced.

 In this paper we study a related question.  Consider a reaction network endowed with stochastic mass action kinetics and let $X_t$ denote the vector whose $i$th coordinate gives the count of species $i$ at time $t$.  We ask the following:  when is it the case that 
  \begin{align}
 \label{eq:initial_PP}
 	P(X_0=x)=  \prod_{i=1}^d e^{-\tilde c_i} \frac{\tilde c_i^{x_i}}{x_i!},  \ \ \text{ for $x \in \Z_{ \ge 0}^d$},
 \end{align}
 where  $\tilde c \in \R^d_{>0}$,
 implies there is a function of time $c:[0,\infty) \to \R^d_{> 0}$  with $c(0) = \tilde c$, for which
 \begin{equation}
 \label{eq:timedep_PP}
 	P(X_t=x)=  \prod_{i=1}^d e^{-c_i(t)} \frac{c_i(t)^{x_i}}{x_i!}, \quad \text{ for all } t \ge 0?
 \end{equation}
 That is, when can the model admit a time dependent distribution that is always a product of Poissons?  Further, when \eqref{eq:timedep_PP} does hold, what is the function $c$?
 
 A partial answer to this question has been  known for quite some time.  In particular, in \cite{gardiner1977poisson} Gardiner showed via \edited{\textit{the Poisson representation}} that if all complexes of the network are either zeroth or first order (which implies linear dynamics), then \eqref{eq:initial_PP} implies \eqref{eq:timedep_PP}  where $c$ is the solution to the associated deterministic model with initial concentration levels given by $c(0)$.  \edited{This result has also been  in \cite{jahnke2007solving}  using direct computations. } 
In this paper, we fully characterize which models have this desired property.  In particular, we introduce a dynamical and restricted (DR) complex balance condition (see Definition \ref{def:DR}), and prove in Theorem \ref{main_theorem} that this is a necessary and sufficient condition for  \eqref{eq:initial_PP} to imply \eqref{eq:timedep_PP}, with $c$ being the solution to the associated deterministic model.

 The outline of the remainder of the paper is as follows.  In Section \ref{sec:model}, we formally introduce the relevant mathematical models, giving the formal definition of a reaction network together with both the stochastic and deterministic model.  We also introduce our new DR condition.  In Section \ref{sec:PR}, we provide our main results, together with their motivation from the physics literature.    In particular, we demonstrate how Gardiner's Poisson Representation (PR), equation \eqref{eq:pr_ansatz}, implies a mathematical conjecture pertaining to which    systems of order two can admit a distribution that is a product of Poissons for all times. We then prove this conjecture while also generalizing to models of order two or higher. In Section \ref{sec:pr_examples}, we provide a series of examples.

\section{Mathematical model}
\label{sec:model}

We formally introduce the mathematical models considered in this paper, together with some key terminology.
\begin{definition}\label{def:network}
A \textit{reaction network} is a triple of finite sets, usually denoted $\{\S,\C,\Re\}$, satisfying the following:
\begin{enumerate}[(i)]
\item the \textit{species}, $\S=\{S_1,\dots,S_d\}$, are the  components whose abundances we wish to model dynamically;
\item the \textit{complexes}, $\C$, are  linear combinations of the species over the nonnegative integers.  Specifically, if $y\in \C$, then
\begin{align}\label{eq:y_LinearComb}
	y = \sum_{i = 1}^d y_i S_i,
\end{align}
with $y_i \in \Z_{\ge 0}$.
\item The \textit{reactions}, $\Re$, are a binary relation on the complexes.  The relation is typically denoted with  ``$\to$'', as in $y \to y'$.

We often enumerate the reactions by $k$, and for $y_k,y_{k}'\in \C$ with $y_k\to y_k' \in \Re$, we call $y_k$ and $y_k'$ the \emph{source} and \emph{product} complexes, respectively, of that reaction.
\end{enumerate}

We  also include the following usual conditions in this definition: every species must appear in at least one complex, every complex must appear as the source or product of at least one reaction, and we do not allow reactions of the type $y \to y \in \Re$ (i.e., we do not allow the source and product complex of a given reaction to be the same).
\end{definition}

Allowing for a slight abuse of notation, we will let $y$ denote both the linear combination of the species, as  in \eqref{eq:y_LinearComb}, and the vector whose $i$th component is $y_i$,  i.e.~$y=(y_1,y_2,\cdots,y_d)^T \in \mathbb{Z}^d_{\ge 0}$.   For example, when $\S=\{S_1,S_2,\dots,S_d\}$, we correspond $2S_1+S_2$ with $(2,1,0,0,\dots,0)^T \in \Z^d_{\ge 0}$. 

For a vector $u\in \R^d$, we let $\|u\|_1 = \sum_{i =1}^d |u_i|$.  We will say that a reaction network is of \textit{first-order} if $\|y\|_1 \leq 1$ for  $\forall y \in \C$, is of \textit{second-order} if $\|y\|_1 \leq 2$ for  $\forall y \in \C$, etc.  For example, the network $4S_1 + S_2 \rightleftarrows 3S_3$ is of 5th-order.

\edited{For a reaction network $\{\S,\C,\Re\}$, it is most commonly represented as a directed reaction graph whose nodes are the complexes and directed edges are given by the reactions. The connected components of the associated reaction graph are termed \textit{linkage classes}. }
A reaction network is said to be \textit{weakly reversible} if for any given reaction, $y\to y'\in \Re$ say, there are reactions, $y_1 \to y_1', \dots, y_\ell \to y_{\ell}'\in \Re$ with $y' = y_1$, $y_i' = y_{i+ 1}$ for each $i \in \{1, \dots, \ell-1\}$, and $y_\ell'= y$.  That is, a model is weakly reversible if \edited{each linkage class} is strongly connected when each complex is written exactly one time. 

When working in a theoretical setting, the set of species is often denoted $\{S_1,\dots, S_d\}$.  However, when working with specific examples one often adopts more suggestive notation such as $E$ for an enzyme, $P$ for a protein, etc.

We provide an example to demonstrate the terminology.

\begin{example}\label{example:8889900}
If in our system we  have only three species, which we denote by $S_1,$ $S_2$, and $S_3$, and the only transition type we allow is the merging of an  $S_1$ and an $S_2$ molecule to form an $S_3$ molecule, then we may depict this network by the directed graph
\begin{equation*}
	S_1 + S_2 \to S_3.
\end{equation*}
For this very simple model our network consists of species $\S = \{S_1,S_2,S_3\}$, complexes $\C = \{S_1+S_2,\ S_3\}$, and reactions $\Re = \{S_1+S_2 \to S_3\}$.  \hfill $\triangle$
\end{example}
%
%

 We now define the two most popular modeling choices for reaction networks: the discrete-space, continuous-time   Markov chain model and the deterministic model.
 
 \vspace{.2in}

\noindent \textbf{Discrete-space, continuous-time Markov chain model.} 
 The usual stochastic model for a reaction network
  treats the system as a continuous-time
Markov chain whose state $X_t\in \Z^d_{\ge 0}$ is a vector whose $i$th component gives the abundance of
species $S_i$ at time $t\ge 0$, and  with each reaction modeled as a
possible transition of the chain.    For the $k$th reaction, we let $y_k\in \Z^d_{\ge0}$ and $y_k'\in \Z^d_{\ge0}$ be the vectors whose $i$th components gives the multiplicity of species $i$ in the source and product complexes, respectively, and let $\lambda_k:\Z^d_{\ge 0} \to \R_{\ge 0}$  give the \emph{transition intensity}, or rate,
at which the reaction occurs.  The transition intensities are often referred to as the \textit{propensities}.
Specifically, if the $k$th reaction occurs at time $t$, then  the old state, $X_{t-}$, is  updated
by addition of the \textit{reaction vector} $\zeta_k = y_k' - y_k$ and
\[
	X_t = X_{t-} + \zeta_k.
\]
  For example, for the
reaction $S_1 + S_2 \to S_3$, we have 
\[
	y_k = \left[ \begin{array}{c}
		1\\
		1\\
		0
		\end{array}\right], \quad  y_k'  = \left[ \begin{array}{c}
		0\\
		0\\
		1
		\end{array}\right],\quad \text{ and }\quad \zeta_k =  \left[ \begin{array}{r}
		-1\\
		-1\\
		1
		\end{array}\right].
		\]
We now assume that $X_t$ is a continuous-time Markov chain on $\Z^d_{\ge 0}$ with transition rates
\[
	q(x,x') = \sum_{k : \zeta_k = x' - x} \lambda_k(x),
\]
where the sum is over all reactions with reaction vector equal to $x' - x$.  The reason for the sum is that different reactions can gave the same reaction vector.  For example, the reactions $S_1 \to S_2$ and $2S_1 \to S_1 + S_2$ have the same reaction vector.  The most common form for the intensity functions $\lambda_k$ is given by \textit{stochastic mass action kinetics}, in which case  
\begin{equation}\label{eq:stochasticMA}
	\lambda_k(x) = \kappa_k \prod_{i=1}^d \frac{x_{i}!}{(x_{i}  - y_{ki})!}1_{\{x_i \ge y_{ki}\}},\quad x \in \Z^d_{\ge 0},
\end{equation}
 where $y_k$ is the source complex and $\kappa_k \in \R_{\ge 0}$ is the rate constant.

\vspace{.1in}

\noindent \textit{Other ways to characterize  the stochastic model.}
The model described above  is a
continuous-time Markov chain in $\Z^d_{\ge 0}$ with infinitesimal generator
\begin{equation}\label{eq:generator}
  ({\mathcal A} f)(x) = \sum_{k} \lambda_k(x)(f(x + \zeta_k) - f(x)),
\end{equation}
where $f : \Z^d \to \R$ \cite{AK2015,Kurtz86}.  Kolmogorov's forward equation, termed the {\em chemical master equation} in much of the biology and chemistry literature,
for this class of models is \cite{AndKurtz2011, anderson2015stochastic,gillespie1992rigorous}
\begin{equation}\label{eq:CME}
  \frac{d}{dt} P_\mu(x, t) = \sum_k \lambda_k(x-\zeta_k)  P_\mu(x-
  \zeta_k,t) 1_{\{x - \zeta_k \in \Z^d_{\ge 0}\}}- \sum_k \lambda_k(x) P_\mu(x,t), 
\end{equation}
where $P_\mu(x,t)$ represents the probability that $X_t = x\in \Z^d_{\ge 0}$, given an initial distribution of $\mu$.  Note that there is one such equation \eqref{eq:CME} for each state in the system (so there are often an infinite number of equations). So long as the process is non-explosive, the different characterizations for the relevant processes  are equivalent \cite{AndKurtz2011,AK2015,Kurtz86}.

\vspace{.2in}

\noindent \textbf{Deterministic model.}  The usual deterministic model with mass action kinetics is the solution to the following ordinary differential equation in $\R^d_{\ge 0}$
\begin{equation}\label{eq:mass-action}
   \frac{d}{dt}{c}(t) = \sum_k \kappa_k  c(t)^{y_{k}}(y_k' - y_k),
\end{equation}
where for two vectors $u,v \in \R^d_{\ge 0}$ we define $u^v \equiv
\prod_i u_i^{v_i}$ and adopt the convention that $0^0=1$. 

\begin{definition}
An equilibrium value $c\in \R_{\ge 0}^d$ is said to be \textit{complex balanced} if for each complex $z\in \C$, 
\[
	\sum_{k: y_k = z} \kappa_k  c^{z} = \sum_{k:y_k^\prime = z} \kappa_k  c^{y_{k}},
\]
where the sum on the left (respectively, right) is over those reactions with source (respectively, product) complex $z$.
\end{definition}

Here we will introduce a new definition, which is closely related to that of a complex balanced equilibrium.   Below and throughout, we denote the 1-norm of a vector $u$ by $\|u\|_1 = \sum_i |u_i|$.
\begin{definition}\label{def:DR}
We say that a solution $c(t)$ to the deterministic dynamics in \eqref{eq:mass-action} satisfies the \textit{dynamical and restricted} (DR, for short) \textit{complex balance condition} if the following holds:   for all complexes $z\in \C$ with $\|z\|_1\ge 2$ and all $t \ge 0$,  
\begin{align}\label{eq:complex_balance}
 	 \sum_{k:y_k = z} \kappa_k c(t)^{z}=\sum_{k:y_k^\prime = z } \kappa_k c(t)^{y_k},
\end{align}
where the sum on the left (respectively, right) is over those reactions with source (respectively, product) complex $z$.
\end{definition} 

\edited{\begin{remark}\label{remark:trivial}
Note that if a reaction network is weakly reversible and if the rate constants are chosen so that the equilibrium concentration $\tilde c$ is  complex balanced, then if we choose $c(0) = \tilde c$ (the complex balanced equilibrium) we  have that $c(t) = \tilde c$ for all $t \ge 0$.  These time-independent solutions are not of interest to us, and we call such solutions  \textit{constant solutions} throughout.\hfill $\triangle$
\end{remark}
}
Thus, the DR conditions is the same as the complex balanced condition except it allows for time dependence (i.e., is dynamical) and is restricted to those complexes that have non-linear intensity functions.  Note that the DR condition holds trivially in the case that $\|z\|_1\le 1$ for all $z \in \mathcal{C}$. An important implication of DR condition is made explicit in Lemma \ref{newlemma}, whose proof is relegated to Appendix \ref{appendix:A}.

\begin{lemma}\label{newlemma}
Consider a reaction network endowed with deterministic mass action kinetics.  Let $c(t)$ be the solution to the system \eqref{eq:mass-action}.  If  for $\tilde c = c(0)\in \R^d_{>0}$ we have that $c(t)$ satisfies the DR condition of Definition \ref{def:DR}, then, for this particular choice of initial condition,  the right-hand side of \eqref{eq:mass-action} is linear and  $c(t) \in \R^d_{>0}$ for all $t \ge 0$.
\end{lemma}

\edited{ The previous lemma gives us one feasible approach to check whether the DR condition holds for a given model. Specifically if the DR condition holds, then by Lemma \ref{newlemma} the system governing the dynamics of $c(t)$ is linear and can therefore be solved explicitly. We can then check whether  the solution so found satisfies the DR condition \eqref{eq:complex_balance}. We will utilize this idea in the following examples and in Section \ref{sec:pr_examples}.  } 

\begin{example}\label{ex_t}
Consider the reaction network with the following network diagram,
\begin{align*}
   2X \xrightleftharpoons[\quad \kappa_2 \quad ]{\kappa_1} 2Y, \quad
   \emptyset  \xrightleftharpoons[\quad \kappa_4 \quad ]{\kappa_3} X, \quad
     \emptyset  \xrightleftharpoons[\quad \kappa_6 \quad ]{\kappa_5} Y,
\end{align*}
where the rate constants are placed next to their respective reaction arrow.
 Notice that $2X$ and $2Y$ are the only complexes that need to be considered in Definition \ref{def:DR}.  The DR condition for both complexes simplifies to the same equation 
\begin{equation}\label{DRcond1}
\kappa_1 x(t)^2  = \kappa_2 y(t)^2
\end{equation}
where $x(t),y(t)$ is the solution to the associated deterministic model \eqref{eq:mass-action}. For the DR condition to be satisfied, we utilize \eqref{DRcond1} in the deterministic model to get
\begin{equation}
\begin{split}\label{ex1_deq}
\frac{dx}{dt} &= -2 \kappa_1 x^2 + 2\kappa_2 y^2 + \kappa_3 - \kappa_4 x  = \kappa_3 - \kappa_4 x,  \qquad \qquad x(0) = x_0  \\
\frac{dy}{dt} &= \hspace{.1in} 2 \kappa_1 x^2 - 2\kappa_2 y^2 + \kappa_5 - \kappa_6 y  = \kappa_5 - \kappa_6 y,  \qquad \qquad y(0) = y_0.
\end{split}\end{equation}
Notice that the system of linear equations \eqref{ex1_deq} has become decoupled, and we can solve them exactly:
\begin{equation}\begin{split}\label{ex1_sol}
x(t) & = \left( x_0 - \frac{\kappa_3}{\kappa_4} \right)e^{-\kappa_4 t}  + \frac{\kappa_3}{\kappa_4} \\
y(t) & = \left( y_0 - \frac{\kappa_5}{\kappa_6} \right)e^{-\kappa_6 t}  + \frac{\kappa_5}{\kappa_6}.
\end{split}\end{equation}
There are two cases to consider.
\begin{enumerate}
\item Suppose $x(0)  = \frac{\kappa_3}{\kappa_4}$.  Then $x(t) = \frac{\kappa_3}{\kappa_4}$ for all time $t\geq 0$. By \eqref{DRcond1},  we must then have
\[
	y(t) =\sqrt{\frac{\kappa_1}{\kappa_2}}  x(t)= \frac{\kappa_3}{\kappa_4}\sqrt{\frac{\kappa_1}{\kappa_2}}.
\]
By \eqref{ex1_sol}, this only holds true if 
$$y_ 0 =  \frac{\kappa_5}{\kappa_6}  = \frac{\kappa_3}{\kappa_4}\sqrt{\frac{\kappa_1}{\kappa_2}}$$
Notice that in this case, both $x(t)$ and $y(t)$ start at complex balanced equilibrium and stay constant for all time $t\geq 0$.   Hence, this case is trivial as noted in Remark \ref{remark:trivial}.  A similar result holds if we had assumed $y_0 = \kappa_5/\kappa_6$.
\item Now suppose that neither $x(t)$ and $y(t)$ start at their complex balanced equilibriums.  By taking the solution \eqref{ex1_sol}, plugging it back into \eqref{DRcond1}, and matching terms, we find that  the rate constants need to satisfy  the following conditions for the DR condition to hold
\begin{align} \label{eq:DRDRDR}
\kappa_4 = \kappa_6, \quad \frac{\sqrt{\kappa_1}}{\sqrt{\kappa_2}} = \frac{\kappa_5}{\kappa_3} = \frac{y_0}{x_0}.
\end{align}
For example, taking 
\[
x_0 = 1, \quad y_0 = 2, \quad \kappa_1 = 4,\quad \kappa_2 = 1,\quad \kappa_3 = 1,\quad \kappa_4 = \frac12,\quad \kappa_5 = 2,\quad \text{and}\quad \kappa_6 = \frac12,
\]
yields the solution
\begin{align*}
	x(t) &= 2 - e^{-t/2}\\
	y(t) &= 4 - 2 e^{-t/2},
\end{align*}
which one can readily check satisfies both the deterministic ODEs \eqref{ex1_deq} and the DR condition \eqref{DRcond1}.  
\end{enumerate}
Hence, if the rate constants and the initial condition satisfy \eqref{eq:DRDRDR}, then the deterministic solution will satisfy the DR condition \eqref{DRcond1}. For other choice of rate constants or initial conditions, there are no non-constant solutions that satisfy DR condition \eqref{DRcond1}. 
\hfill $\triangle$
\end{example}

\begin{example}\label{ex_n}
Consider the network 
\begin{equation*}\begin{split}
   X  \xrightleftharpoons[\quad \kappa_2 \quad ]{\kappa_1} 2Y, \quad
   \emptyset  \xrightleftharpoons[\quad \kappa_4 \quad ]{\kappa_3} X, \quad
    \emptyset  \xrightleftharpoons[\quad \kappa_6 \quad ]{\kappa_5} Y,
\end{split}
\end{equation*}
where the rate constants have been placed next to their respective reactions.
Note that this model is weakly reversible, and there is therefore a choice of rate constants for which it is complex balanced.  For this model, the DR condition of Definition \ref{def:DR} is
\begin{equation}\label{DRcond3}
\kappa_1 x(t)  = \kappa_2 y(t)^2
\end{equation}
where $x(t)$ and $y(t)$ are the solutions to the associated deterministic model \eqref{eq:mass-action}. To see when the DR conditions is satisfied, we utilize \eqref{DRcond3} in the deterministic model  to get
\begin{equation}
\begin{split}\label{ex3_deq}
\frac{dx}{dt} &= - \kappa_1 x + \kappa_2 y^2 + \kappa_3 - \kappa_4 x  = \kappa_3 - \kappa_4 x  \qquad \qquad x(0) = x_0  \\
\frac{dy}{dt} &= \hspace{.1in} 2 \kappa_1 x - 2\kappa_2 y^2 + \kappa_5 - \kappa_6 y  = \kappa_5 - \kappa_6 y  \qquad \qquad y(0) = y_0.
\end{split}\end{equation}
Notice that the system of linear equation \eqref{ex3_deq} is exactly the same as the system \eqref{ex1_deq}, and we have
\begin{equation}\begin{split}\label{ex3_sol}
x(t) &= \left( x_0 - \frac{\kappa_3}{\kappa_4} \right)e^{-\kappa_4 t}  + \frac{\kappa_3}{\kappa_4} \\
y(t) &= \left( y_0 - \frac{\kappa_5}{\kappa_6} \right)e^{-\kappa_6 t}  + \frac{\kappa_5}{\kappa_6}. 
\end{split}\end{equation}

We will now demonstrate that there is not choice of parameters, except in the trivial case, that will satisfy \eqref{DRcond3}.  As before, there are two cases that need consideration.
\begin{enumerate}
\item Suppose $x(0)  = \frac{\kappa_3}{\kappa_4}$.  Then $x(t) = \frac{\kappa_3}{\kappa_4}$ for all time $t\geq 0$. By \eqref{DRcond3}, we must then have
$$y(t) =\sqrt{\frac{\kappa_1}{\kappa_2}x(t)} = \sqrt{\frac{\kappa_1\kappa_3}{\kappa_2\kappa_4}}. $$
By \eqref{ex3_sol}, the above only holds true if 
$$y_ 0 =  \frac{\kappa_5}{\kappa_6}  =  \sqrt{\frac{\kappa_1\kappa_3}{\kappa_2\kappa_4}}   $$
Notice that in this case, both $x(t)$ and $y(t)$ start at complex balanced equilibrium and stay constant for all time $t\geq 0$.  Hence, this is the trivial case discussed in Remark \ref{remark:trivial}.  A similar result is found if one assumes first that $y_0 = \frac{\kappa_5}{\kappa_6}$.
\item Suppose now that neither $x(t)$ nor $y(t)$ starts at its equilibrium.  We then take the solution \eqref{ex3_sol} and plug it back into \eqref{DRcond3}, yielding
\begin{align*}
\kappa_1 \left( \left( x_0 - \frac{\kappa_3}{\kappa_4} \right)e^{-\kappa_4 t}  + \frac{\kappa_3}{\kappa_4} \right) & = \kappa_2 \left( y_0 - \frac{\kappa_5}{\kappa_6} \right)^2 e^{-2\kappa_6 t}+ 2\kappa_2  \frac{\kappa_5}{\kappa_6} \left( y_0 - \frac{\kappa_5}{\kappa_6} \right)e^{-\kappa_6 t} + \kappa_2 \frac{\kappa_5^2}{\kappa_6^2}.
\end{align*}
The key observation is that in order to balance the three exponential terms, one of them \textit{must} have a coefficient that is zero.  However, this would imply that we are back in case 1. 
\end{enumerate}
Hence, there are no non-constant solutions which satisfy DR condition \eqref{DRcond3}. 
\hfill $\triangle$
\end{example}

\section{Motivation and results}
\label{sec:PR}

\subsection{Motivation from the physics literature}

In the physics literature, there is an alternative representation for the solution to the chemical master equation \eqref{eq:CME} and is given by Gardiner's \emph{Poisson representation} (PR) \cite{gardiner1977poisson}.   \edited{We will present this representation here, and show a conjecture it implies, since they  served as the motivation for the present work.}

One form of the PR (the ``positive PR'' \cite{gardiner1985stochastic}) can be derived by first making the following ansatz for $P_\mu(x,t)$ from \eqref{eq:CME}:
\begin{align}\label{eq:pr_ansatz}
  P_{\mu}(x,t)
  & =
    \int_{\mathbb{C}^d} \prod_{i=1}^d \mathcal{P}(x_i; u_i)   \pi_{\nu}(u,t) d u , \quad u=(u_1, \ldots, u_d),
\end{align}
where  $\mathcal{P}(x_i;u_i)=(e^{-u_i}u_i^{x_i})/x_i!$ is a Poisson distribution in $x_i$ with mean $u_i$, and where $\pi_\nu(\cdot, \cdot)$ is a function on $\mathbb{C}^d\times \R_{\ge 0}$ satisfying $\pi_\nu(u,0) = \nu(u)$. Note that the integrals in  \eqref{eq:pr_ansatz}  are taken over the whole complex plane for each $u_i$. Under certain conditions one can use the ansatz  \eqref{eq:pr_ansatz}, together with the chemical master equation  \eqref{eq:CME}, to derive an evolution equation for $\pi_{\nu}(u,t)$  \cite{gardiner1977poisson}.  
Specifically, under the further assumption that for each complex $y \in \C$ we have $\|y\|_1 \le 2$ (i.e.~the system is binary),
  one can formally derive that $\pi_{\nu}(u,t)$ fulfills the Fokker-Planck equation \cite{gardiner1977poisson}
\begin{align}\label{eq:pr_fpe}
  \frac{\partial}{\partial t} \pi_{\nu}(u,t)
  & = 
    - \sum_{i=1}^d \frac{\partial}{\partial u_i} \left[ A_i (u) \pi_{\nu}(u, t) \right]
    + \frac{1}{2} \sum_{i,j=1}^d \frac{\partial}{\partial u_i} \frac{\partial}{\partial u_j}
     \left[ B_{ij}(u) \pi_{\nu}(u, t) \right],
\end{align}
with drift vector $A(u)$ and diffusion matrix $B(u)$ given by
\begin{align}
\label{pr_drift}
  A_i(u)
  & = 
    \sum_k \kappa_k u^{y_k} \zeta_{ki}, \\
\label{pr_diff}
  B_{ij}(u)
  & =
    \sum_k \kappa_k u^{y_k} (y'_{ki} y'_{kj}- y_{ki}y_{kj}- \delta_{i,j} \zeta_{ki}),
\end{align}
where $\delta_{i,j} $ denotes the Kronecker delta, and where the initial condition is $\pi_\nu(u,0) = \nu(u)$.

Now suppose that $B(u) \equiv 0$ and that the initial condition satisfies $\nu(u) = \delta(u-u^0),$ i.e.~is the Dirac delta function, for some constant $u^0 \in \Z^d_{\ge 0}$.  Note that, from \eqref{eq:pr_ansatz}, having $\nu(u) = \delta(u-u_0)$ corresponds to a product of Poissons for an initial distribution of the process $X_t$, i.e.~$P_{\mu}(x,0) = \mu(x) = \prod_{i=1}^d \mathcal{P}(x_i; u^0_i)$.  Now note that because $B(u) \equiv 0$  the equation for $\pi_\nu$ in \eqref{eq:pr_fpe} reduces to a Liouville equation and $\pi_{\nu}$ remains a delta distribution for all times centered around the deterministic process $u(t)$, which fulfills the ordinary differential equation \eqref{eq:mass-action}. This means that $X_t$ has  a distribution given by a product of Poissons  for all times: $P_{\mu}(x,t)  = \prod_{i=1}^d \mathcal{P}(x_i; u_i(t))$.

Collecting thoughts, we have shown that the PR representation in the physics literature implies the following conjecture. 

\begin{conjecture}\label{conjecture}
  Suppose that the following three conditions hold:
  \begin{enumerate}[(i)]
  \item the reaction network is binary, i.e.~$\|y\|_1\le 2$ for each complex, 
  \item  the initial distribution of the stochastically modeled reaction network is a product of Poissons,
  \item  $B(u(t)) = 0$,  where $u(t)$ solves \eqref{eq:mass-action} and $B$ is as in \eqref{pr_diff}.
  \end{enumerate}
  Then the distribution of the process $X_t$ is a product of Poissons for all time.  
  \end{conjecture}
  
Note that we trivially have  $B(u) =0$ for all $u$ if the model is \edited{first-order} (i.e.~if $\|y\|_1 \le 1$ for each $y \in \C$).  

\edited{In the remaining sections}, we will show that Conjecture \ref{conjecture} is correct.    \edited{In fact, we do more:  we derive \textit{necessary and sufficient conditions} that characterize when a model can admit a distribution that is a product of Poissons for all time.  However, we explicitly point out here that we will do so without using the Poisson representation of \eqref{eq:pr_ansatz} or the Fokker-Planck equation \eqref{eq:pr_fpe}, as \eqref{eq:pr_fpe} only follows from \eqref{eq:pr_ansatz} under heuristic methods that, to the best of our knowledge, are not mathematically justified.}



\subsection{Main results}

Our main result, Theorem \ref{main_theorem},  shows that a stochastically modeled reaction network has a product-form distribution for all time if and only if the initial distribution is a product of Poissons and the DR condition from Definition \ref{def:DR} holds for the associated deterministic model.

\begin{theorem}\label{main_theorem}
Consider a stochastically modeled reaction network with  intensity functions given by stochastic mass action kinetics \eqref{eq:stochasticMA}.     Suppose that $X_0$ has a distribution that is a product of Poissons, i.e. there is a $\tilde{c} \in \R_{> 0}^d$ for which
\begin{align}\label{initial_distribution}
\mu (x) = \prod_{i=1}^d  e^{-\tilde{c}_i}\frac{\tilde{c}_i^{x_i}}{x_i ! }, \ \ \text{  for $x \in \Z_{\ge 0}^d$},
\end{align}
where $\mu(x) =  P_\mu(X_0 = x)$.
Then the following three statements are equivalent: 
\begin{enumerate}[(i)]
\item the solution to the  ODE \eqref{eq:mass-action} with $c(0) = \tilde c$ satisfies the DR condition of Definition \ref{def:DR};

\item  the solution to the chemical master equation $P_\mu(x,t)$ satisfies 
\begin{align}\label{poissonAT}
P_\mu(x,t) =  \prod_{i=1}^d  e^{-c_i(t)}\frac{c_i(t)^{x_i}}{x_i ! }, \ \  \text{        for $x\in \Z_{\ge 0}^d$ and all $t \ge 0$},
\end{align}
for some deterministic process $c(t)$ with $c(0)  = \tilde c$;

  \item  the solution to the chemical master equation $P_\mu(x,t)$ satisfies 
\begin{align}\label{7697679}
P_\mu(x,t) =  \prod_{i=1}^d  e^{-c_i(t)}\frac{c_i(t)^{x_i}}{x_i ! }, \ \  \text{        for $x\in \Z_{\ge 0}^d$ and all $t \ge 0$},
\end{align}
for $c(t)$ satisfying \eqref{eq:mass-action} with $c(0) = \tilde c$.
\end{enumerate}
\end{theorem}

\edited{\begin{remark}\label{remark:trivialprob}
Similarly as in Remark \ref{remark:trivial}, if we choose $c(0) = \tilde c$ (the complex balanced equilibrium) we  have that $c(t) = \tilde c$ for all $t \ge 0$ and that \eqref{7697679} also holds for all time (with $c(t) = \tilde c$) and is the stationary distribution of the stochastic model. However, these time-independent solutions are not of interest to us, and we call such solutions  \textit{constant solutions} throughout.\hfill $\triangle$
\end{remark}}

\edited{
\begin{remark}\label{remark2}
By Theorem \ref{main_theorem} above, a model satisfying the DR condition has a distribution satisfying \eqref{7697679}.  If we also have that $\lim_{t\to \infty} c(t) = C \in \R^d_{>0}$, then the model has a stationary distribution of the form \cite{AndProdForm}
\[
	\prod_{i = 1}^d e^{-C_i} \frac{C_i^{x_i}}{x_i!}, \ \ \text{ for $x \in \Z^d_{\ge 0}$}.
\]
Therefore, by results in \cite{CW2016}, the model must be complex balanced, with complex balanced equilibrium $C$.  Hence, in this case the model satisfies both the DR condition and the complex balancing condition.  Of course, this logic does not hold if there is an $i \in \{1,\dots, d\}$ for which $\lim_{t\to \infty} c_i(t) \in \{0,\infty\}$. 
\hfill $\triangle$
\end{remark}}

Before proving Theorem \ref{main_theorem}, we note that  the next logical question would be: when will the DR condition hold? The following lemma answers this question for binary networks:  the DR condition holds if and only if $B(u(t)) = 0$ where $u(t)$ solves the ODE \eqref{eq:mass-action}.

\begin{lemma}\label{lemma}
Consider a binary reaction network, i.e.~$\|y\|_1\le 2$ for all $y \in \C$.  Then the DR condition from Definition \ref{def:DR} holds for the associated deterministic model \eqref{eq:mass-action} if and only if $B(u(t)) = 0$ with $u(t)$ satisfying \eqref{eq:mass-action}. 
\end{lemma}

Note that taken together, Theorem \ref{main_theorem} and Lemma \ref{lemma} show that Conjecture \ref{conjecture} stated in the previous section holds.

\begin{proof}[Proof of Lemma \ref{lemma}]
First note that if $\|z\|_1 \le 1$ for all $z \in \mathcal{C}$, then both conditions hold.  We may therefore consider the case where $\|z\|_1\le 2$ for each $z \in \mathcal{C}$ and $\|z\|_1=2$ for at least one complex $z \in \mathcal{C}$.

First, let us rewrite the expression in the parentheses of $B(u)$ in \eqref{pr_diff} as 
\begin{align*}
y_{ki}^\prime y_{kj}^\prime - y_{ki} y_{kj} - \delta_{ij} \zeta_{ki} = f_{ij}(y_k^\prime) - f_{ij}(y_k)\quad  \text{ where } \quad f_{ij} (y_k) = y_{ki} y_{kj} - \delta_{ij} y_{ki}.
\end{align*}
It is straightforward to show that for given indices $i$ and $j$, the expression $f_{ij}(y_k)$ is non-zero if and only if $y_k = e_i + e_j$, where $e_i$ denotes the vector with the $i^{th}$ entry equal to $1$ and zero otherwise. This means we can rewrite $B_{ij}(u)$ as 
\begin{align}
B_{ij}(u(t)) &= \sum_{k} \kappa_k u(t)^{y_k} (f_{ij}(y_k^\prime) - f_{ij}(y_k) )  \notag\\ 
& = \sum_{k:y_k^\prime = e_i+e_j} \kappa_k u(t)^{y_k} f_{ij}(e_i + e_j) -  \sum_{k:y_k = e_i+e_j} \kappa_k u(t)^{y_k} f_{ij}(e_i + e_j)\notag  \\
& = f_{ij } (e_i+e_j) \left(  \sum_{k: y_k^\prime = e_i+e_j} \kappa_k u(t)^{y_k} - \sum_{k: y_k = e_i+e_j} \kappa_k u(t)^{y_k} \right) .\label{eq:87966}
\end{align}
where the first sum is over those reactions with product complex $e_i + e_j$ and the second sum is over those reactions with source complex $e_i + e_j$. Since each $f_{ij}(e_i + e_j) > 0$, we see that $B(u(t)) = 0$ if and only if the term in parentheses in \eqref{eq:87966} is zero for each choice of $i$ and $j$.  The equivalence of the two conditions then follows. 
\end{proof}

The following proposition will be of use.
\begin{prop}\label{theorem_c}
Consider a stochastically modeled reaction network with  intensity functions given by stochastic mass action kinetics \eqref{eq:stochasticMA}.   Suppose there is a deterministic function $c(t)$, defined for $t \ge 0$, for which $P_\mu(x,t)$, the solution to the Kolmogorov forward equation \eqref{eq:CME}, satisfies \eqref{poissonAT}. Then, $E[X(t)] = c(t)$ is  the solution to the deterministic equation  \eqref{eq:mass-action} with $\tilde c = c(0)$.
\end{prop}
\begin{proof}
The infinitesimal generator of the continous-time markov chain model is the operator $\mathcal{A}$ given by \eqref{eq:generator}.
Since the distribution of $X(t)$ is given by \eqref{poissonAT}, we know that $\EE[X_i(t)] = c_i(t)$. Moreover,
\begin{align}
\begin{split}
\EE[&\lambda_k(X(s))] = \kappa_k \EE\left[ \frac{X(s)!}{(X(s) - y_k)!}\right] \\
 &= \kappa_k\EE \left[ \prod_{i=1}^d \frac{X_i(s)!}{(X_i(s) - y_{ki})!} \right] = \kappa_k\sum_{x\in \Z_{\ge 0}^{d}}   \prod_{i=1}^d  \frac{x_i!}{(x_i - y_{ki})!}  \prod_{i=1}^d  e^{-c_i(s)}\frac{c_i(s)^{x_i}}{x_i ! } \\
& =\kappa_k \sum_{x\in \Z_{\ge 0}^{d}}  \prod_{i=1}^d  e^{-c_i(s)}\frac{c_i(s)^{x_i}}{(x_i - y_{ki})!}  = \kappa_kc(s)^{y_k} \sum_{x\in \Z_{\ge 0}^{d}}  \prod_{i=1}^d  e^{-c_i(s)}\frac{c_i(s)^{x_i-y_{ki}}}{(x_i - y_{ki})!}  = \kappa_kc(s)^{y_k},
\end{split}
\label{eq:7686}
\end{align}
where the final equality holds since we are summing a probability mass function over all of $\Z^d_{\ge 0}$.
For $m >0$,  applying Dynkin's formula with the function $f_m(x) = x_i \wedge m \equiv \min\{x_i,m\}$ yields
\begin{align*}
\EE[ X_i(t)\wedge m] &= \EE[X(0)\wedge m] + \EE\left[ \int_0^t  \mathcal{A} f_m(X(s)) ds\right] \\
& = \EE[X_i(0)\wedge m] + \int_0^t \EE\left[\sum_{k=1}^K \lambda_k(X(s) ) ((X_i(s) +\zeta_{ki}) \wedge m - X_i(s)\wedge m) \right] ds. 
\end{align*}
 Noting that $\sup_{x\in \Z^d_{\ge 0}}  |(x_i + \zeta_{ki})\wedge m - x_i\wedge m| \le \max_\ell \|\zeta_{\ell}\|_{\infty}$ for all $i$, we may let $m\to \infty$ and apply the Dominated convergence theorem to conclude
\begin{align}\label{dynkinnew}
\EE[ X(t)] & = \EE[X(0)] + \int_0^t \EE\left[\sum_{k=1}^K \lambda_k(X(s) ) \zeta_k \right] ds. 
\end{align}
Combining  \eqref{dynkinnew} with  \eqref{eq:7686}, together with the fact that $c(t) = E[X(t)]$, yields
\begin{align*}
c(t) &= \tilde c + \int_0^t \sum_{k=1}^K \kappa_k c(s)^{y_k} \zeta_k ds. 
\end{align*}
Differentiating both sides shows that $c(t)$ is the solution to \eqref{eq:mass-action}.
\end{proof}

We now turn to the proof of Theorem \ref{main_theorem}.  We begin by stating two technical lemmas whose proofs are relegated to  Appendix \ref{appendix:B}.

\begin{lemma}\label{Calculation}
Suppose $P_\mu(x,t)$ is given by \eqref{poissonAT} with $c(t) \in \R_{>0}^d$ for all $t \ge 0$.   Then $P_\mu(x,t)$ is the solution to the Kolmogorov forward equation  \eqref{eq:CME} if and only if $c(t)$ satisfies the deterministic equation \eqref{eq:mass-action} and 
\begin{align}\label{eq:eq}
 \sum_{k} \kappa_k c(t)^{y_k}  \left[  g_{x,c(t)}(y_k^\prime) -g_{x,c(t)}(y_k)  \right] = 0 
\end{align}
where for each $x \in \Z^d_{\ge 0}$ and $c \in \R^d_{ > 0}$,
\begin{align}\label{eq:ComplicatedFunction}
g_{x,c}(y_k) =  \sum_{j=1}^d \left( \frac{x_j}{c_j}-1\right) y_{kj} - \frac{x !}{(x -y_{k})!} c^{-y_k} +1.
\end{align}
Moreover, if $\|y_k\|_1 \leq 1$, then $g_{x,c} (y_k) = 0$. 
\end{lemma} 

\begin{lemma}\label{Function_Independent}
Let $\{z_1,z_2, ...., z_m\} \subset \C$ be the collection of complexes that are at least binary (i.e.~$\|z_i\|_1 \ge 2$). Fix a value $c \in \R^d_{>0}$.  For each $i \in \{1,\dots, m\}$ let $f_i: \Z_{\ge 0}^d \rightarrow \R$ be defined as
\[
	f_i(x) = g_{x,c}(z_i).
\]
Then $\{ f_i\}_{i=1}^m$ are linearly independent \edited{as functions of x}. 
\end{lemma}

We now prove Theorem \ref{main_theorem}.

\begin{proof}[Proof of Theorem \ref{main_theorem}]

First note that the implication \textit{(iii) $\implies$ (ii)} is trivial.  We will now show that \textit{(ii) $\implies$ (i)} and that \textit{(i) $\implies$ (iii)}.

\vspace{.1in}

\noindent Proof that  \textit{(ii) $\implies$ (i)}. 

By proposition \ref{theorem_c}, $E[X(t)] = c(t)$ solves the  deterministic equation  \eqref{eq:mass-action} with $\tilde c = c(0) \in \R^d_{>0}$.  
Therefore, we just need to show that $c(t)$ will satisfy the DR condition of Definition \ref{def:DR} .  Since there is always a positive probability that no reaction takes place by time $t>0$, we know that $E[X_i(t)] = c_i(t) > 0$.   Hence, because $P_\mu(x,t)$ defined in \eqref{poissonAT} is the solution to the chemical master equation \eqref{eq:CME}, Lemma \ref{Calculation} allows us to conclude that  \eqref{eq:eq}  holds with $g_{x,c}(y)$ defined as in \eqref{eq:ComplicatedFunction}. Since $g_{x,c(t)}(z) = 0$ if  $\|z\|_1 \le 1$, we can rewrite  \eqref{eq:eq} as a summation over complexes which are at least binary: 
\begin{align*}
\sum_{z : \|z\|_1\ge 2}  g_{x,c(t)} (z) \left[  \sum_{k: y_k^\prime = z} \kappa_k c(t)^{y_k} - \sum_{k: y_k = z} \kappa_k c(t)^{y_k} \right] = 0. 
\end{align*}
Because the above holds for all $x\in \Z^d_{\ge 0}$, Lemma \ref{Function_Independent} allows us to conclude that each term in brackets is identically equal to zero:
$$\sum_{k: y_k^\prime = z} \kappa_k c(t)^{y_k} =  \sum_{k: y_k = z} \kappa_k c(t)^{y_k}, $$
which is exactly the the DR condition of Definition \ref{def:DR}.

\vspace{.1in}

\noindent Proof that  \textit{(i) $\implies$ (iii)}. 

Suppose that for $c(t)$ satisfying the ODE \eqref{eq:mass-action}  we have 
\begin{align*}
 \sum_{k: y_k^\prime = z } \kappa_k c(t)^{y_k}  =  \sum_{k: y_k = z} \kappa_k c(t)^{y_k},
\end{align*}
for those $z$ with $\|z\|_1 \ge 2$.
Then for any $x\in \Z^{d}_{\ge 0}$ we may multiply the above by the functions $g_{x,c(t)}(z)$ defined in \eqref{eq:ComplicatedFunction} and conclude 
\begin{align*}
g_{x,c(t)}(z)  \sum_{k: y_k^\prime = z } \kappa_k c(t)^{y_k}  = g_{x,c(t)}(z) \sum_{k: y_k = z} \kappa_k c(t)^{y_k}.
\end{align*}
Note that the previous step is valid since $c(t)  \in \R^d_{> 0}$ by Lemma \ref{newlemma}.
We now sum over all complexes $z$ (not just those with $\|z\|_1 \ge 2$), while noting that $g_{x,c(t)}(z) = 0$ if $\|z\|_1 \le 1$, to see
\begin{align*}
 0 & = \sum_{z} g_{x,c(t)}(z)  \left( \sum_{k:y_k^\prime = z } \kappa_k c(t)^{y_k}  - \sum_{k:y_k = z} \kappa_k c(t)^{y_k}  \right) \\
 & = \sum_{k=1}^K \kappa_k c(t)^{y_k}    \left( g_{x,c(t)}(y_k^\prime) -g_{x,c(t)}(y_k)  \right).
\end{align*}
which, by Lemma \ref{Calculation}, implies $P_\mu(x,t)$ in \eqref{poissonAT} is the solution to the chemical master equation. Uniqueness of the solution to the chemical master equation follows from Lemma 1.23 in \cite{anderson2015stochastic}.
\end{proof}

\section{Examples}
\label{sec:pr_examples}

We provide a number of examples to demonstrate our theory.
We first provide two non-first order examples that satisfy the DR condition, and hence admit a time dependent distribution that is a product of Poissons.   These examples will  make it clear that satisfying the DR condition is difficult in that the parameters and initial conditions of the model must be chosen precisely.   \edited{Example \ref{ex:linearNoDR} is then provided to demonstrate that even when a model admits an effectively linear deterministic system, the associated stochastic system still may not satisfy the DR condition.}  Next, we provide two examples, Examples \ref{ex_n_2} and \ref{xplusy}, which demonstrate   that in the time-dependent case there exist networks for which \textit{no choice of rate constants} will yield a model that satisfies the DR condition (except in the trivial case--see Remark \ref{remark:trivialprob}--when the initial condition is equal to a complex balanced equilibrium). \edited{Finally, Example \ref{example5} is included to facilitate the understanding of the proof of Lemma \ref{newlemma} and Example \ref{example6} shows that the DR condition does not imply weak reversibility of any portion of the network, which is different from the classical theory of complex balanced models. }

%
%
%
%
%
%
%

\begin{example}
Consider the reaction network in Example \ref{ex_t},
\begin{align*}
   2X \xrightleftharpoons[\quad \kappa_2 \quad ]{\kappa_1} 2Y, \quad
   \emptyset  \xrightleftharpoons[\quad \kappa_4 \quad ]{\kappa_3} X, \quad
     \emptyset  \xrightleftharpoons[\quad \kappa_6 \quad ]{\kappa_5} Y,
\end{align*}
where the rate constants are placed next to their respective reaction arrow. Now, by Example \ref{ex_t} and Theorem \ref{main_theorem}, if the rate constants and the initial condition satisfy \eqref{eq:DRDRDR}, then  for any $z \in \Z^2_{\ge 0}$ and $t \ge 0$,
\begin{align*}
	P_\mu(z,t) &= e^{-(x(t) + y(t))} \frac{x(t)^{z_1}}{z_1!} \frac{y(t)^{z_2}}{z_2!}.
\end{align*}

 A few remarks are in order.  First, note that for this example the diffusion matrix $B$ from \eqref{pr_diff} is
 \begin{align*}
 B(u) = 
    \begin{pmatrix}
      -2 \kappa_1 u^2 + 2 \kappa_2 u_2^2  &  0 \\
       0 & 2 \kappa_1 u_1^2 - 2 \kappa_2 u_2^2  \\
    \end{pmatrix},
 \end{align*}
which also yields the  equation  \eqref{DRcond1} when we set $B((x(t),y(t))) = 0$. 

Second, this model will admit a complex balanced equilibrium if and only if
\[
	\frac{\sqrt{\kappa_1}}{\sqrt{\kappa_2}} = \frac{\kappa_4}{\kappa_6}\cdot \frac{\kappa_5}{\kappa_3},
\]
which is a \textit{less} restrictive condition on the parameters of the model than \eqref{eq:DRDRDR}.  Said differently, there are choices of rate constants (for example when $\kappa_4 \ne \kappa_6$) for which the underlying model is complex balanced, but for which the DR condition does not hold.
  \hfill $\triangle$
\end{example}

For some choices of rate constants, the previous model admitted a positive complex balanced equilibrium.    The next example shows that 
a time dependent distribution that is a product of Poissons may still exist even if the associated deterministic model admits no positive equilibria for any choice of rate constants.

\begin{example}\label{ex_nont}
\edited{Consider the \textit{decaying-dimerization reaction set} which was introduced in \cite{gillespie2001approximate},}
\begin{align*}
   X \xrightleftharpoons[\quad \kappa_2 \quad ]{\kappa_1 } 2Y, \quad
   X  \xrightarrow{\quad \kappa_3  \quad } Z,\quad
   Y  \xrightarrow{\quad \kappa_4  \quad } \emptyset. 
\end{align*}
\edited{Note that,  because of the reaction $Y\to \emptyset$, as $t \to \infty$ the deterministic and stochastic models will both converge to the boundary of $\R^d_{\ge 0}$ with $x  = y = 0$.}

The DR condition of Definition \ref{def:DR} is
\begin{equation}\label{DRcond2}
\kappa_1 x(t)  = \kappa_2 y(t)^2
\end{equation}
where $x(t)$ and $y(t)$ are the solutions to the associated deterministic model \eqref{eq:mass-action}.  We search for solutions that satisfy the DR condition by plugging \eqref{DRcond2} into the deterministic model \eqref{eq:mass-action}
\edited{
\begin{equation}
\begin{split}\label{ex2_deq}
\frac{dx}{dt} &= - \kappa_1x+ \kappa_2 y^2 - \kappa_3 x= -\kappa_3 x \qquad \qquad   \hspace{0.1in} x(0) = x_0 \\
\frac{dy}{dt} &=  \hspace{.1in}2 \kappa_1 x - 2\kappa_2 y^2 - \kappa_4 y = -\kappa_4 y \qquad \qquad y(0) = y_0 \\
\frac{dz}{dt} &=  \hspace{.1in} \kappa_3 x  \hspace{1.6in}\qquad \qquad z(0) = z_0.
\end{split}\end{equation}}
As in the previous example, the system of equations \eqref{ex2_deq} can be solved exactly yielding a solution of 
\edited{
\begin{equation}\label{ex2_sol}
x(t) = x_0 e^{- \kappa_3 t} \qquad \qquad \qquad y(t) = y_0 e^{-\kappa_4 t} \qquad \qquad \qquad z(t) = z_0 + x_0 (1-e^{-\kappa_3 t}).
\end{equation}}
Requiring that \eqref{DRcond2} holds enforces the following conditions
\begin{equation}\label{eq:;lkdajf;lkjd;fja}
	\kappa_3 = 2\kappa_4\quad \text{and} \quad \frac{\kappa_1}{\kappa_2} = \frac{y_0^2}{x_0}.
\end{equation}
Hence, any model satisfying the conditions \eqref{eq:;lkdajf;lkjd;fja} will yield a distribution satisfying \eqref{7697679}.  

For example, suppose we have
\edited{
\begin{equation}\label{eq:rateconstant}
	\kappa_1 = 9, \quad \kappa_2 = 1,\quad  \kappa_3 = 2, \quad \kappa_4 = 1, \quad x_0 = 900, \quad \text{and}\quad y_0 = 90 \quad \text{and}\quad z_0 = 100.
\end{equation}}
Then the solution to \eqref{eq:mass-action} is
\edited{
\begin{equation}\label{eq:sdfsbdf}
	x(t) = e^{-2t}, \quad y(t) = 3e^{-t},\quad z(t) = 2 - e^{-2t}
\end{equation}}
which can be readily checked to satisfy the DR condition \eqref{DRcond2}.

Hence, by Theorem \ref{main_theorem} we have that for any \edited{$w \in \Z^3_{\ge 0}$} and $t \ge 0$,
\edited{
\begin{align*}
	P_\mu(w,t) &= e^{-(x(t) + y(t)+z(t))} \frac{x(t)^{w_1}}{w_1!} \cdot \frac{y(t)^{w_2}}{w_2!} \cdot \frac{z(t)^{w_3}}{w_3!}.
\end{align*}}
Note that even though \edited{$2X(t) + Y(t)+Z(t)$ only decreases along the trajectory, i.e. that $2X(t) + Y(t)+Z(t) \leq 2x_0 + y_0+z_0$ for all $t \ge 0$, the relevant state space is still all of \edited{$\Z^3_{\ge 0}$} as our initial distribution is the product of Poissons
}
\edited{
\[
	\mu(w) = e^{-(x_0+ y_0+z_0)} \frac{x_0^{w_1}}{w_1!}\cdot  \frac{y_0^{w_2}}{w_2!}\cdot  \frac{z_0^{w_3}}{w_3!},
\]}
 which has support on all of \edited{$\Z^3_{\ge 0}$}. \edited{We performed numerical experiments on this model and present their results in Figure \ref{fig:senmod4}. }
\begin{figure}
\begin{center}
\includegraphics[width = 3in]{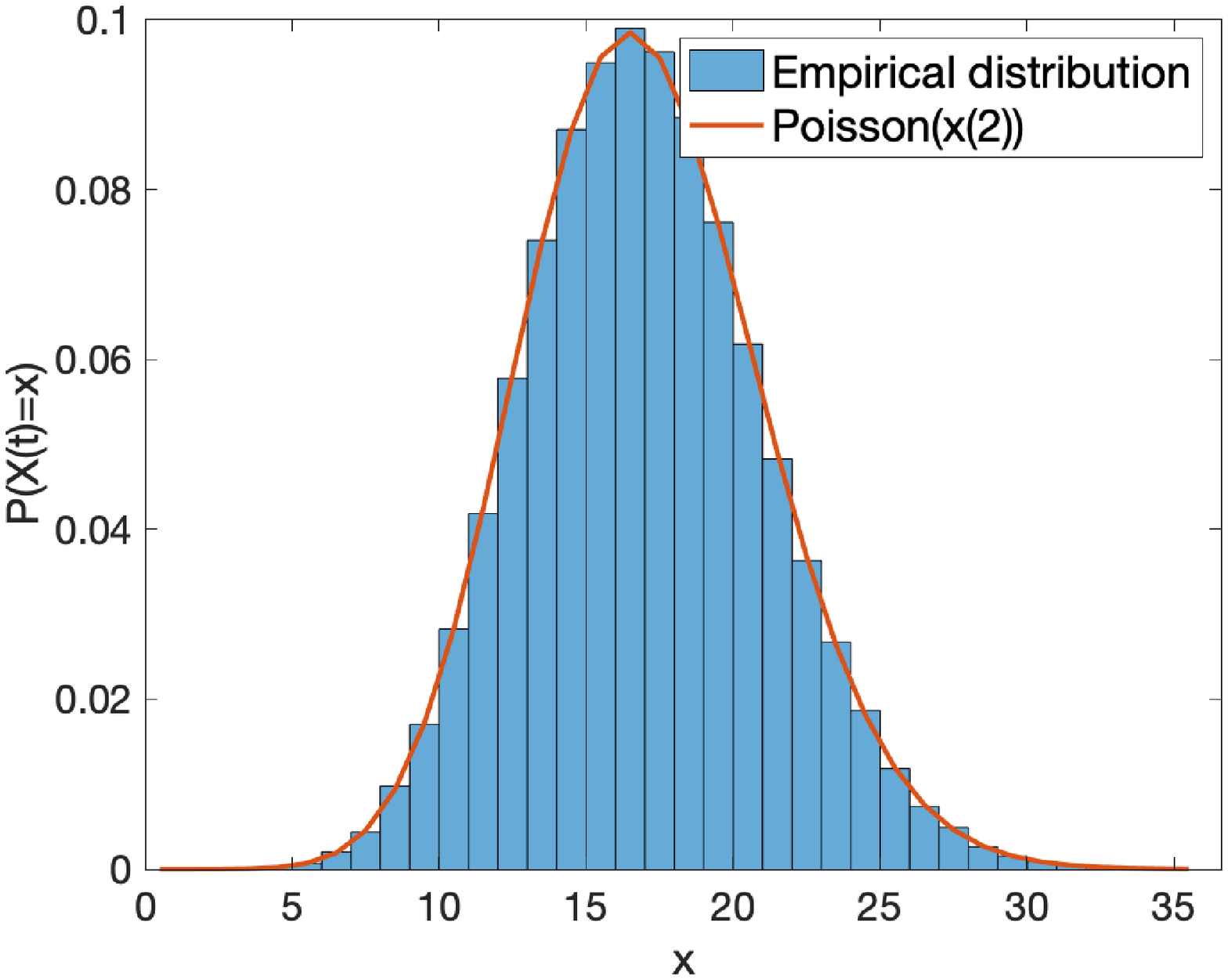} \includegraphics[width = 3in]{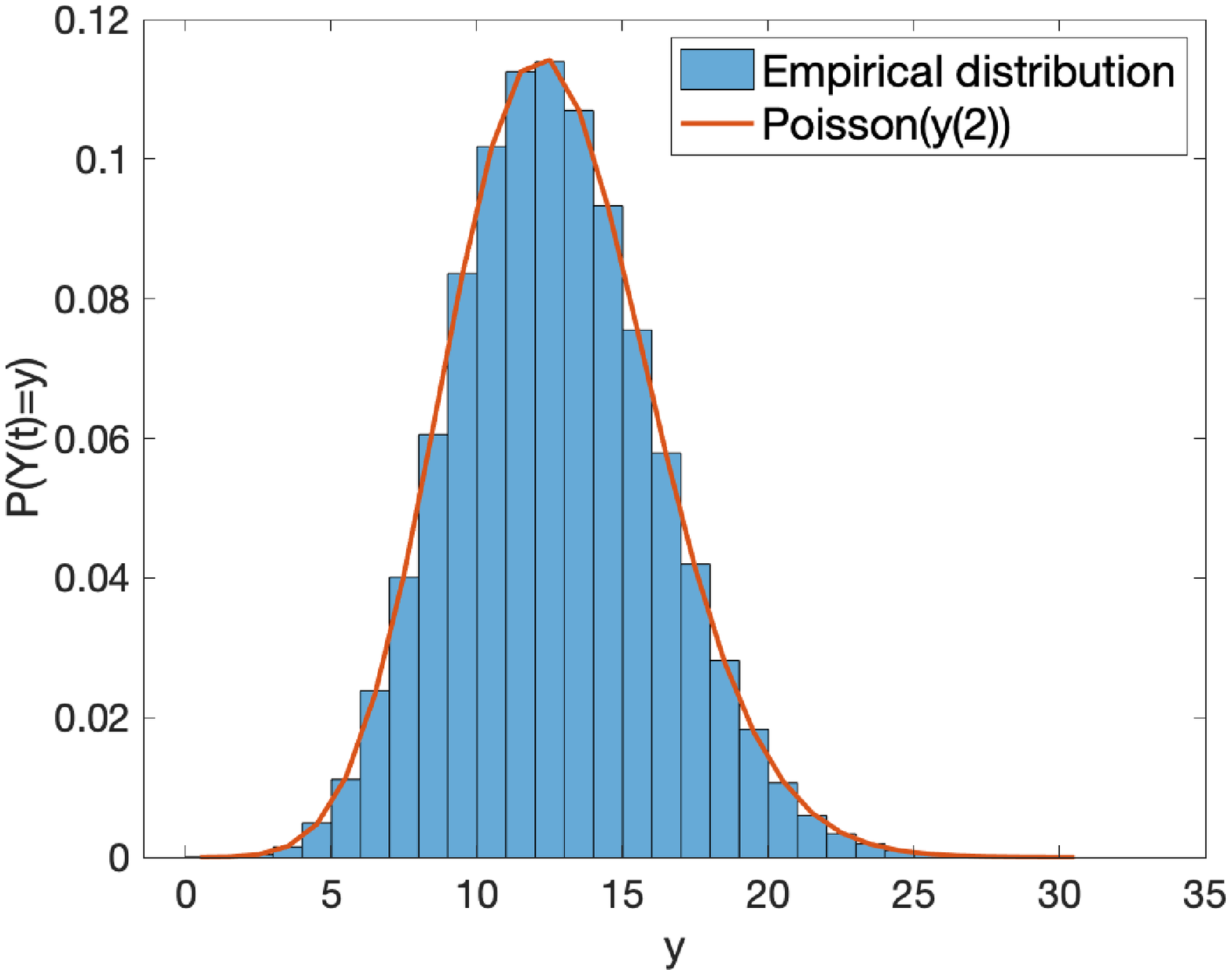}
\end{center}
\caption{Comparison of an empirical probability mass function and exact probability distribution of species $X$ and $Y$ at time $T = 2$ for Example \ref{ex_nont}, where the parameters of the model are given in \eqref{eq:rateconstant}.  The empirical probability mass functions were obtained using Monte Carlo  with $N = 10^6$ trajectories and is plotted via a histogram. The exact probability distribution is provided by Theorem \ref{main_theorem}, with $c(t)$ given by \eqref{eq:sdfsbdf}.}
\label{fig:senmod4}
\end{figure}
\hfill $\triangle$
\end{example}

\begin{example}\label{ex:linearNoDR}
\edited{Consider the network with the following network diagram,
\begin{align*}
   X \xrightleftharpoons[\quad \kappa_2 \quad ]{\kappa_1 } 2Y, \quad
   X  \xrightarrow{\quad \kappa_3  \quad } Z,\quad
   Y  \xrightarrow{\quad 4\kappa_4  \quad } \emptyset, \quad 
   Y  \xrightarrow{\quad \kappa_4  \quad } 4Y. 
\end{align*}
Note that the DR condition can not be satisfied for complex $4Y$ since it is not a source complex for any reaction. However, this model was specifically chosen so that the dynamics of the associated deterministic system are the same as \eqref{ex2_deq} in Example \ref{ex_nont}:  
\begin{equation}
\begin{split}\label{exadd_deq}
\frac{dx}{dt} &= - \kappa_1x+ \kappa_2 y^2 - \kappa_3 x  \qquad \qquad   \hspace{1.2in} x(0) = x_0 \\
\frac{dy}{dt} & =  2 \kappa_1 x - 2\kappa_2 y^2 - \kappa_4y \hspace{1.18in} \qquad \qquad y(0) = y_0 \\
\frac{dz}{dt} &=   \kappa_3 x  \hspace{2.25in}\qquad \qquad z(0) = z_0.
\end{split}\end{equation}
Hence, if parameters are chosen satisfying \eqref{eq:;lkdajf;lkjd;fja}, the solution to \eqref{exadd_deq} is given by \eqref{ex2_sol} and the dynamics are effectively linear. }

\edited{
However, even though the dynamics are effectively linear,  the DR condition does not hold and Theorem \ref{main_theorem} tells us that the time evolution of the master equation can not be solved as a time-dependent product-form Poisson distribution. We verified this numerically by performing  simulations on the model with parameters given via \eqref{eq:rateconstant}.  The results are presented in Figure \ref{fig:senmod5}. The resulting distributions are clearly non-Poissonian.  Moreover, the empirical mean and variance of $X$ at time $2$ are  given by 
\begin{equation*}
\EE[X(2)] \approx 16.63 \neq  57.78 \approx \text{Var}(X(2)).
\end{equation*}
In conclusion, we see that even effectively linear dynamics does not guarantee a time dependent Product-form Poisson distribution.
\begin{figure}
\begin{center}
\includegraphics[width = 3in]{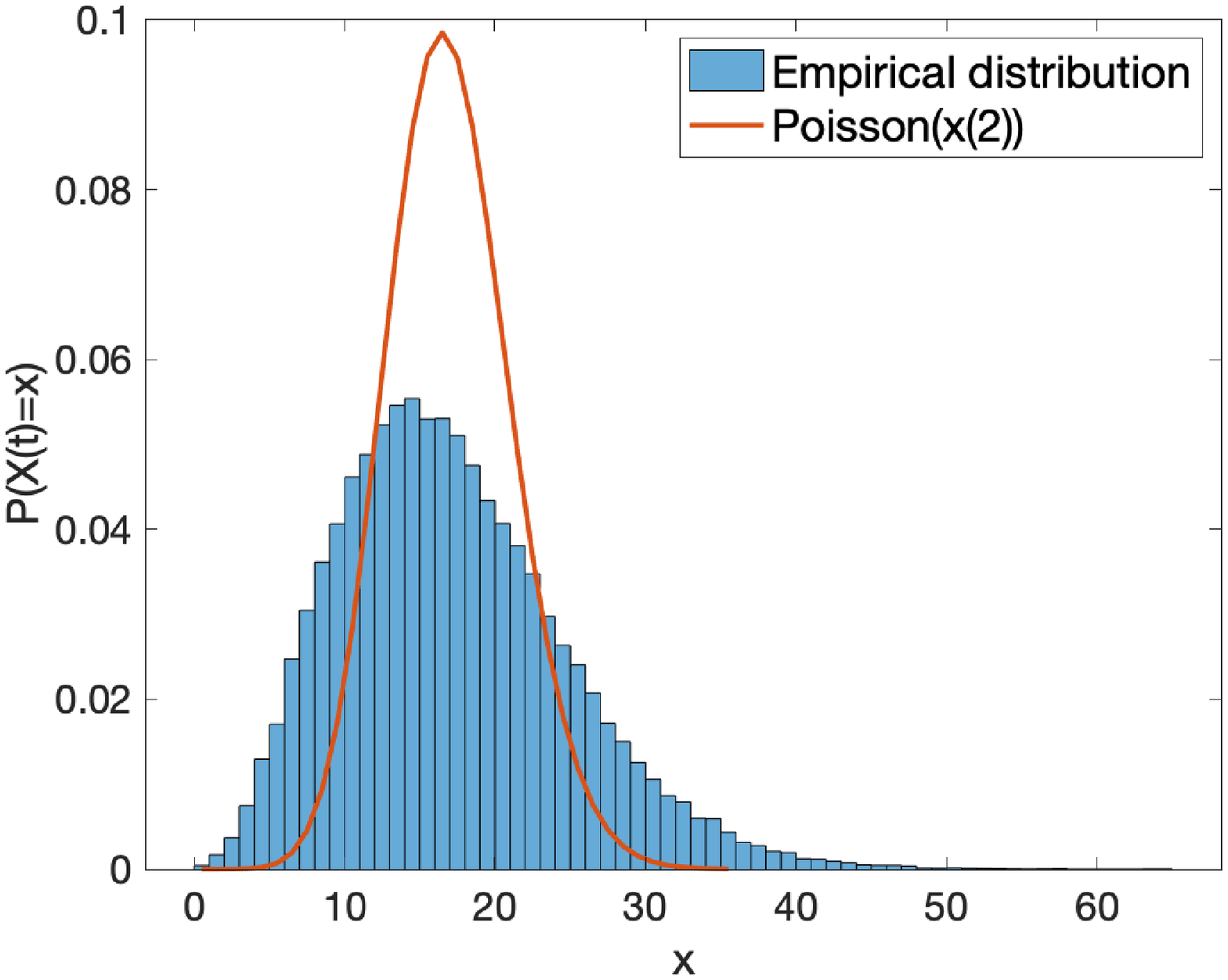} \includegraphics[width = 3in]{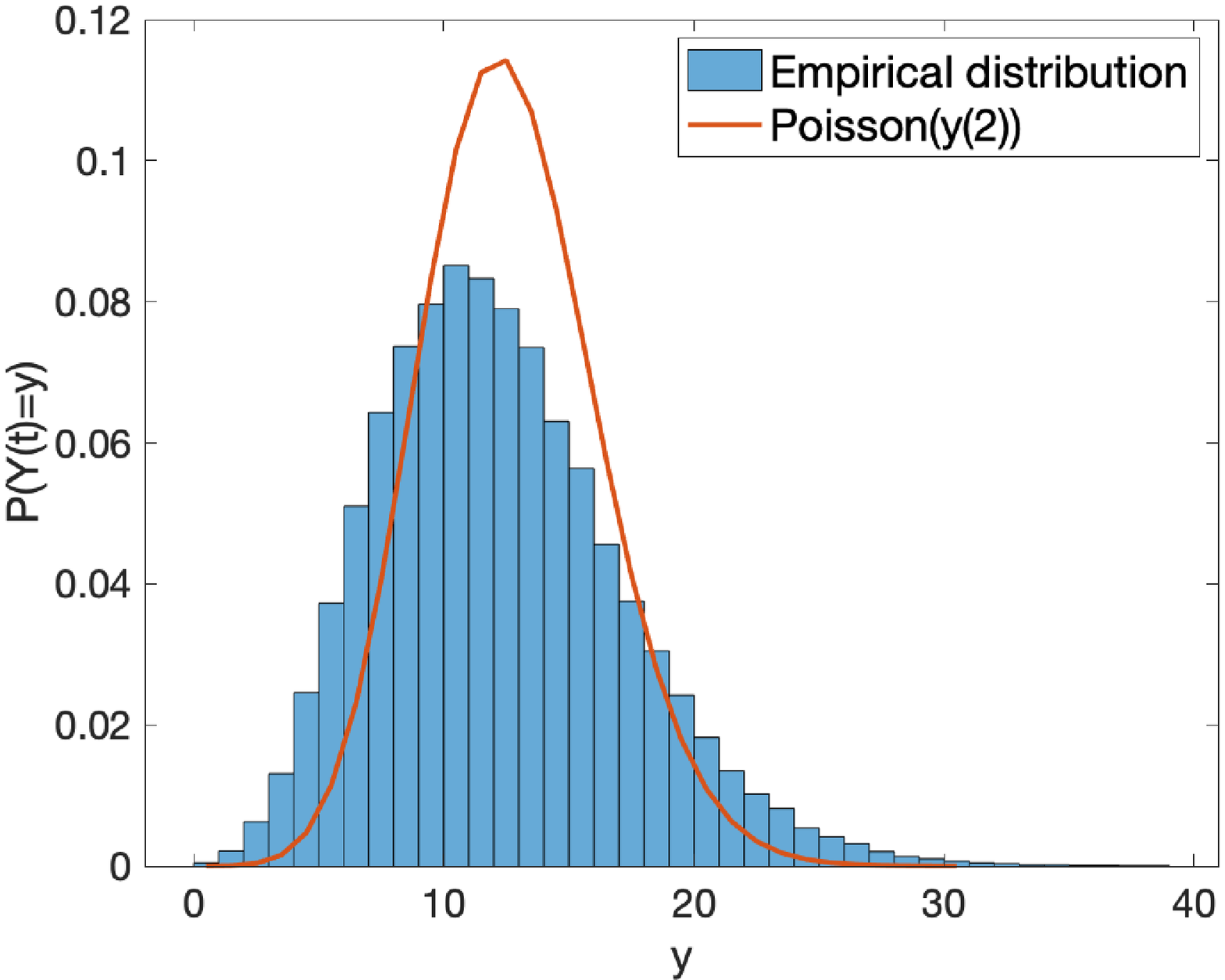}
\end{center}
\caption{Comparison of an empirical probability mass function and  Poisson distribution  of species $X$ and $Y$ at time $T = 2$ for Example \ref{ex:linearNoDR}, where the parameters of the model are given in \eqref{eq:rateconstant}.  The empirical probability mass functions were obtained using Monte Carlo  with $N = 10^6$ trajectories and is plotted via a histogram. The  Poisson distributions were chosen to have the same means as in Example  \ref{ex_nont} since the ODE models for the two examples are the same.
}
\label{fig:senmod5}
\end{figure}}
\hfill $\triangle$
\end{example}

For any weakly reversible model, there exists a choice of rate constants that make the resulting model complex balanced \cite{horn1972necessary}. The next two examples demonstrate that there are weakly reversible networks for which no nontrivial (in the sense of Remark \ref{remark:trivial}) solution to the forward equation is a product of Poissons, regardless of the choice of rate constants.

\begin{example}\label{ex_n_2}
Consider the network in Example \ref{ex_n},
\begin{equation*}\begin{split}
   X  \xrightleftharpoons[\quad \kappa_2 \quad ]{\kappa_1} 2Y, \quad
   \emptyset  \xrightleftharpoons[\quad \kappa_4 \quad ]{\kappa_3} X, \quad
    \emptyset  \xrightleftharpoons[\quad \kappa_6 \quad ]{\kappa_5} Y,
\end{split}
\end{equation*}
where the rate constants have been placed next to their respective reactions. By Example \ref{ex_n}, we may conclude that no nonconstant solution exists and, by Theorem \ref{main_theorem}, there is no choice of parameters which yields a distribution that is a product of Poissons for all time.\hfill $\triangle$
\end{example}

\begin{example}\label{xplusy}
Consider the network 
\begin{align*}
   \emptyset \xrightleftharpoons[\quad \kappa_2 \quad ]{\kappa_1} X+Y, \quad
   \emptyset  \xrightleftharpoons[\quad \kappa_4 \quad ]{\kappa_3} X, \quad
     \emptyset  \xrightleftharpoons[\quad \kappa_6 \quad ]{\kappa_5} Y, \quad 
    X  \xrightleftharpoons[\quad \kappa_8 \quad ]{\kappa_7} Y,  
\end{align*}
and assume that $\kappa_1, \kappa_2>0$.  We will show that this model can not satisfy the DR condition of Definition \ref{def:DR} for any choice of rate constants.

First note that for this model the DR condition  reduces to
\begin{equation}\label{DRcond4}
\kappa_1  = \kappa_2 x(t)  y(t)  \quad  \iff   \quad x(t)   = \frac{\kappa_1}{\kappa_2 } y(t)^{-1},
\end{equation}
where $x(t)$ and $y(t)$ are the solutions to the associated deterministic model \eqref{eq:mass-action}, and we are assuming that $y(t)>0$ for all $t \ge 0$. Assuming the DR condition holds, the associated deterministic model is 
\begin{equation}
\begin{split}\label{ex4_deq}
\frac{dx}{dt} &=  \kappa_3  + \kappa_8 y-  (\kappa_4 +\kappa_7)x \qquad \qquad x(0) = x_0  \\
\frac{dy}{dt} &= \kappa_5  +\kappa_7 x - (\kappa_6 +\kappa_8) y   \qquad \qquad y(0) = y_0.
\end{split}\end{equation}
Instead of solving this system explicitly, which leads to quite a messy solution, we note that \eqref{DRcond4} implies  
\[
	\frac{dx}{dt} = -\frac{\kappa_1}{\kappa_2} y^{-2} \frac{dy}{dt}.
\]
Plugging \eqref{ex4_deq} into the above equation yields
\[
  \kappa_3  + \kappa_8 y-  (\kappa_4 +\kappa_7)x =  -\frac{\kappa_1}{\kappa_2} y^{-2} \left( \kappa_5  +\kappa_7 x - (\kappa_6 +\kappa_8) y \right),
\]
which, after again using that we must have $x = \frac{\kappa_1}{\kappa_2}y^{-1}$ due to \eqref{DRcond4}, becomes
\[
 \kappa_3  + \kappa_8 y- (\kappa_4 + \kappa_7) \frac{\kappa_1 }{\kappa_2} y^{-1}   =   -\kappa_5 \frac{\kappa_1}{\kappa_2} y^{-2}  - \kappa_7\frac{\kappa_1^2}{\kappa_2^2} y^{-3} + (\kappa_6 +\kappa_8)\frac{\kappa_1}{\kappa_2} y^{-1}
\]
or
\[
\kappa_3 y^3  + \kappa_8 y^2- \left[ (\kappa_4 + \kappa_7) \frac{\kappa_1 }{\kappa_2} + (\kappa_6 +\kappa_8)\frac{\kappa_1}{\kappa_2} \right] y^{2}     + \kappa_5 \frac{\kappa_1}{\kappa_2}  y + \kappa_7\frac{\kappa_1^2}{\kappa_2^2}=0.
\]
We have assumed that $y(t)$ is a nonconstant solution of the system, so the equation above implies the associated polynomial has an infinite number of roots.  Of course, this can not be as a third degree polynomial  has at most 3 roots.  Hence, we may conclude that each of the coefficients of the above polynomial must be zero.  Combining this fact with the assumption that $\kappa_1, \kappa_2>0$ we find
\[
	\kappa_3 = \kappa_4 = \kappa_5 = \kappa_6 = \kappa_7 = \kappa_8 = 0.
\]
Hence, the only possibility is if the entire network is $\emptyset \xrightleftharpoons[\quad \kappa_2 \quad ]{\kappa_1}  X + Y$.  However, then there can not be a nonconstant solution that satisfies the DR condition as $\kappa_1 = \kappa_2 x(t) y(t)$ implies that $x(t),y(t)$ is at equilibrium (thereby yielding a constant solution).\hfill $\triangle$
\end{example}

The logic at the end of the previous example can be used to characterize all one-dimensional models that satisfy the DR condition.

\begin{prop}\label{DR1species}
Consider a reaction network $\{\S,\C,\Re\}$ with one species, i.e., $\|\S\| = 1$ and suppose that the initial distribution of the associated Markov model satisfies \eqref{initial_distribution}.  Then the solution to the forward equation \eqref{eq:CME} is given by \eqref{7697679} for some nontrivial process $c(t)$ if and only if the reaction network is of first order, in which case $\C = \{\emptyset, X\}$.
\end{prop}

\begin{proof}
Of course, if the system is first order, then the DR condition automatically holds and Theorem \ref{main_theorem} implies that  the solution to the forward equation \eqref{eq:CME} is given by \eqref{7697679}.

We now show the other direction, and the proof will proceed by contradiction.  Thus, suppose that there is a complex of the form $z = k X$ for some $k \ge 2$, and suppose that the solution to the forward equation \eqref{eq:CME} is given by \eqref{7697679} for some nontrivial process $c(t)$.   By Theorem \ref{main_theorem}, we may assume that the solution to the deterministic model \eqref{eq:mass-action} satisfies the DR condition of Definition \ref{def:DR} for the complex $z$.  That is, 
\begin{align*}
 \sum_{k:y_k = z} \kappa_k c(t)^{z}=\sum_{k:y_k^\prime = z} \kappa_k c(t)^{y_k},
\end{align*}
where, as usual, the sum on the left is over those reactions with source complex $z$ and the sum on the right is over those with product complex $z$.  Consider the function
\[
f(x) =  \sum_{k:y_k = z} \kappa_k x^{\|z\|_1} - \sum_{k:y_k^\prime = z} \kappa_k x^{\|y_k\|_1}.
\]
Note that $f$ is a polynomial in $x$.  Also, each sum is nonempty and, because $\|y_k\|_1 \ne \|z\|_1$ for each term in the second sum, $f$ is not identically equal to zero. Thus, $f$ has a finite number of roots.  However, $f(c(t)) = 0$, and $c(t)$ is nontrivial, implying $f$ has an infinite number of roots, which is a contradiction.  Thus, the result is shown.
\end{proof} 

\edited{
The next example will demonstrate how a key piece of the proof of Lemma \ref{newlemma} will proceed. In particular, we will assume the DR condition holds, and will then conclude that the nonlinear terms from the higher-order monomials can be written as a linear combination of the first-order monomials.  We will then be able to conclude that no non-constant solution to the rate equations exist that satisfies the DR condition.}
\begin{example}\label{example5}
\edited{Consider the reaction network with the following network diagram,
\begin{align}
   X \xrightleftharpoons[\quad \kappa_2 \quad ]{\kappa_1} 2X + Y \xrightleftharpoons[\quad \kappa_4 \quad ]{\kappa_3} X+2Y \xrightleftharpoons[\quad \kappa_6 \quad ]{\kappa_5} Y,
     \end{align}
where the rate constants are placed next to their respective reaction arrows. Notice that the DR condition \ref{def:DR} for the complexes $2X+Y$ and $X+2Y$ can be simplified to the equations
 \begin{align}
 \begin{split} \label{DRcond598988}
 (\kappa_2 + \kappa_3) x(t)^2 y(t) &= \kappa_1 x(t) +\kappa_4 x(t) y(t)^2, \\ 
 (\kappa_4+\kappa_5) x(t) y(t)^2 &= \kappa_3 x(t)^2 y(t)+ \kappa_6 y (t), 
 \end{split}
 \end{align}
 respectively, where $(x(t) ,y(t))$ is the solution to the associated deterministic model \eqref{eq:mass-action}, 
\begin{equation}
\begin{split}\label{ex5_deq}
\frac{dx}{dt} &= \kappa_1 x - ( \kappa_2 +\kappa_3 )x^2 y + ( \kappa_4 - \kappa_5) xy^2  + \kappa_6 y,  \qquad \qquad x(0) = x_0  \\
\frac{dy}{dt} &=  \kappa_1 x - ( \kappa_2 - \kappa_3 )x^2 y - ( \kappa_4 + \kappa_5) xy^2 + \kappa_6 y,  \qquad \qquad y(0) = y_0.
\end{split}\end{equation}
At first glance the resulting dynamics appear nonlinear, since these higher order monomials do not cancel out immediately when \eqref{DRcond598988} is used. Nevertheless, we can rewrite the DR condition \eqref{DRcond598988} as follows
 \begin{align*}
(\kappa_2 + \kappa_3) x(t)^2 y(t)  - \quad \qquad \kappa_4   x(t) y(t)^2 &= \kappa_1 x(t), \\ 
- \kappa_3 x(t)^2 y(t) + (\kappa_4+\kappa_5) x(t) y(t)^2 &=  \kappa_6 y (t),
 \end{align*}
and notice that it is in the form of a vector equation $A \tilde x = b$, where 
\[
A = 
    \begin{bmatrix}
      \kappa_2 + \kappa_3  & - \kappa_4 \\
       -\kappa_3 & \kappa_4+\kappa_5 \\
    \end{bmatrix}, \qquad 
\tilde x = \begin{bmatrix}
	x(t)^2 y(t)  \\
	x(t) y(t)^2 \\
	\end{bmatrix}, \qquad 
b = \begin{bmatrix}
	\kappa_1 x(t)  \\
	\kappa_6 y (t) \\
	\end{bmatrix}.
\]
It is easy to observe that $A$ is nonsingular, and its inverse can be calculated as
\[
A^{-1} = \frac{1}{(\kappa_2+\kappa_3)(\kappa_4+\kappa_5) - \kappa_3 \kappa_4}
    \begin{bmatrix}
      \kappa_4 + \kappa_5 &  \kappa_4 \\
        \kappa_3 & \kappa_2 + \kappa_3 \\
    \end{bmatrix}.
\]
Hence $\tilde x$ can be written as 
\begin{align}\label{eq:DRcond5}
\tilde x  = \begin{bmatrix}
	x(t)^2 y(t)  \\
	x(t) y(t)^2 \\
	\end{bmatrix} 
= A^{-1} b = \frac{1}{(\kappa_2+\kappa_3)(\kappa_4+\kappa_5) - \kappa_3 \kappa_4} 
	\begin{bmatrix}
	(\kappa_4 + \kappa_5)\kappa_1 x(t) + \kappa_4 \kappa_6 y (t)	\\
	\kappa_1\kappa_3 x(t) + (\kappa_2 + \kappa_3) \kappa_6 y(t) \\
	\end{bmatrix}. 
\end{align}
Therefore, assuming the DR condition holds, we may represent the higher order monomials as a linear combination of first order monomials. Plugging \eqref{eq:DRcond5} back into the ODE \eqref{ex5_deq}, we get
\begin{align*}
\frac{dx}{dt} &= \kappa_1 x - ( \kappa_2 +\kappa_3 )\frac{(\kappa_4 + \kappa_5)\kappa_1 x + \kappa_4 \kappa_6 y }{(\kappa_2+\kappa_3)(\kappa_4+\kappa_5) - \kappa_3 \kappa_4}  + ( \kappa_4 - \kappa_5)\frac{\kappa_1\kappa_3 x + (\kappa_2 + \kappa_3) \kappa_6 y}{(\kappa_2+\kappa_3)(\kappa_4+\kappa_5) - \kappa_3 \kappa_4}  + \kappa_6 y \\
& = - \frac{\kappa_1 \kappa_3\kappa_5}{(\kappa_2+\kappa_3)(\kappa_4+\kappa_5) - \kappa_3 \kappa_4} x +  \frac{\kappa_2 \kappa_4 \kappa_6}{(\kappa_2+\kappa_3)(\kappa_4+\kappa_5) - \kappa_3 \kappa_4} y,  \\
\frac{dy}{dt} &= \kappa_1 x - ( \kappa_2 - \kappa_3 )\frac{(\kappa_4 + \kappa_5)\kappa_1 x + \kappa_4 \kappa_6 y }{(\kappa_2+\kappa_3)(\kappa_4+\kappa_5) - \kappa_3 \kappa_4}  -   ( \kappa_4 + \kappa_5) \frac{\kappa_1\kappa_3 x + (\kappa_2 + \kappa_3) \kappa_6 y}{(\kappa_2+\kappa_3)(\kappa_4+\kappa_5) - \kappa_3 \kappa_4}  + \kappa_6 y \\
& =  \frac{\kappa_1 \kappa_3\kappa_5}{(\kappa_2+\kappa_3)(\kappa_4+\kappa_5) - \kappa_3 \kappa_4} x -  \frac{\kappa_2 \kappa_4 \kappa_6}{(\kappa_2+\kappa_3)(\kappa_4+\kappa_5) - \kappa_3 \kappa_4} y,
\end{align*}
where we arrive at a linear model.  Also notice that  $\frac{dx}{dt} + \frac{dy}{dt} = 0$, and so we must have $x(t)+y(t) = x_0 + y_0$ for all $t\geq 0$. \\
}

\edited{
However, there is no non-constant solution to \eqref{ex5_deq} satisfying the DR condition \eqref{eq:DRcond5}. Specifically, the DR condition for complex $2X+Y$ becomes 
\begin{align*}
(\kappa_2 + \kappa_3) x(t)^2 (x_0 +y_0 - x(t)) &= \kappa_1 x(t) +\kappa_4 x(t) (x_0+y_0-x(t) )^2. 
\end{align*}
By virtue of the proof of Proposition \ref{DR1species}, we equate the coefficients on both sides and we get $\kappa_1 = \kappa_2 = \kappa_3 = \kappa_4 = 0$. Similarly using the DR condition for $X + 2Y$, we get $\kappa_5 = \kappa_6 = 0$. 
\hfill $\triangle$
}
\end{example}

\edited{
It is known that complex balanced models are necessarily weakly reversible \cite{F1}. Since the DR condition implies complex balancing for all higher order complexes, one may expect that part of the network to be  weakly reversible. However, the next example  shows that this claim is incorrect. 
\begin{example}\label{example6}
Consider the network with the following  diagram,
\begin{align*}
   Z  \xrightarrow{ \quad 2  \quad  } 2X \xrightarrow{\quad 2  \quad } 2Y \xrightarrow{ \quad 2  \quad  }  W\quad,\quad X \xrightarrow{ \quad 1  \quad  } \emptyset \quad,  \quad Y \xrightarrow{ \quad 1  \quad  } \emptyset. 
\end{align*}
The DR condition \ref{def:DR} for the complexes $2X$ and $2Y$ can be simplified to the equations
 \begin{align}
 \begin{split} \label{DRcond5989818}
2 z(t) =2 x^2(t) \qquad \qquad 2 x^2(t) = 2 y^2(t)
 \end{split}
 \end{align}
where $(x(t),y(t),z(t),w(t))$ is the solution to the associated deterministic model \eqref{eq:mass-action}. For the DR condition to be satisfied, we utilize \eqref{DRcond5989818} in the deterministic model to get
\begin{equation}
\begin{split}\label{ex8_deq}
\frac{dx}{dt} &=  2z(t) - 2x^2(t) - x(t) = -x(t),  \qquad \qquad x(0) = x_0  \\
\frac{dy}{dt} &=  2x^2(t) - 2y^2(t) - y(t) = -y(t),  \qquad \qquad y(0) = y_0 \\
\frac{dz}{dt} &=  -2 z(t),  \qquad \qquad \hspace{1.5in} z(0) = z_0  \\
\frac{dw}{dt} &=  2y^2(t),  \qquad \qquad \hspace{1.5in} w(0) = w_0. \\
\end{split}\end{equation}
Notice that the system of linear equations \eqref{ex8_deq} can be solved exactly for $x(t),y(t),w(t)$, and hence $z(t)$, with
\begin{equation*}
x(t) = x_0 e^{-t} , \quad y(t) = y_0 e^{-t}   , \quad z(t) =  z_0 e^{-2t} , \quad w(t) = w_0 + y_0(1-e^{-2t}). 
\end{equation*}
Hence the DR condition \eqref{DRcond5989818} holds if and only if
\begin{align*}
z_0 = x_0^2 = y_0^2. 
\end{align*}
However, note that no portion of this network, nor any of its subnetworks, are weakly reversible. 
\hfill $\triangle$
\end{example}}

\section{Acknowledgements}
We would like to thank the Isaac Newton Institute for hosting a 6 month program entitled ``Stochastic Dynamical Systems in Biology: Numerical Methods and Applications'' where this collaboration initiated.  Anderson and Yuan are currently supported by Army Research Office grant W911NF-18-1-0324. Schnoerr is currently supported by Biotechnology and Biological Sciences Research Council grant BB/P028306/1.

\appendix


\section{Proof of Lemma \ref{newlemma}}\label{appendix:A}

The proof will proceed in a manner similar to that of Example \ref{example5}, in that we will show that under the assumption that the DR condition holds, the non-linear monomials can be written as a linear combination of the linear terms.  
In order to make this precise, we require a number of definitions.  

The $i$th row of matrix $A$ is said to be \textit{strictly diagonally dominant} (SDD) if $|a_{ii}| >  \sum_{j\neq i} |a_{ij}|$. We then say that the matrix $A$ is \textit{strictly diagonally dominant} if all its rows are SDD. Similarly, the $i$th row of matrix $A$ is said to be \textit{weakly diagonally dominant} (WDD) if $|a_{ii}| \geq \sum_{j\neq i} |a_{ij}|$ and we say that the matrix $A$ is \textit{weakly diagonally dominant} if all its rows are WDD. 

There is a directed graph associated to any $m\times m$ square matrix. Its vertices are given by $\displaystyle \{1,2,....,m\}$ and its edges are defined as follows:  for $i \ne j$, there exists an edge $i\rightarrow j$ if and only if $a_{ij} \ne 0$. 

SDD matrices are always invertible \cite{horn2012matrix}. However, WDD matrices could be singular and the following lemma can  be used to identify invertibility of a WDD matrix \cite{shivakumar1974sufficient}. 

\begin{lemma}\label{lemma:WCDD}
Suppose that $A$ is WDD and that for each row $i_1$ that is not SDD, there exists a walk $ i_1 \rightarrow i_2 \rightarrow \ldots \rightarrow i_k$ in the directed graph of $A$ ending at row $i_k$, which is SDD.   Then $A$ is non-singular.
\end{lemma}
 
\vspace{.1in}

%

We restate Lemma \ref{newlemma} for the sake of reference.
\vspace{.1in}

\noindent \textbf{Lemma 2.1} \textit{
Consider a reaction network endowed with deterministic mass action kinetics.  Let $c(t)$ be the solution to the system \eqref{eq:mass-action}.  If  for $\tilde c = c(0)\in \R^d_{>0}$ we have that $c(t)$ satisfies the DR condition of Definition \ref{def:DR}, then, for this particular choice of initial condition,  the right-hand side of \eqref{eq:mass-action} is linear and  $c(t) \in \R^d_{>0}$ for all $t \ge 0$.
}
\vspace{.1in}

\begin{proof} 

We begin by noting that some  deterministic models may blow-up in finite time.  We therefore define  
\begin{align*}
	T^* =\inf \{ s  > 0 : \text{for any $m >0$, there exists $\varepsilon > 0$ such that when $s- t \leq \varepsilon$,  $\| c(t) \|_1 \geq m$ holds}\}.
	\end{align*}
	  Note that if the set is empty, then we take $T^*$ to be infinity.
	  Our first goal will be to show that $c(t) \in \R^d_{>0}$ for any $t < T^*$.  
	  
	  We therefore let $t<T^*$.  We then know that there exists an $m>0$, such that $\|c(s)\|_1\leq m$ for any $s\in [0,t]$.  Consider the $i$th component of the differential equation with $s \ge t$:
	  \begin{align*}
   \frac{d}{ds}{c_i}(s) &= \sum_k \kappa_k  c(s)^{y_{k}}(y_k' - y_k) \geq \sum_{k:y_{ki}^\prime - y_{ki}<0} \kappa_k  c(s)^{y_{k}}(y_{ki}' - y_{ki}) \\
   & \geq c_i(s) \sum_{k:y_{ki}^\prime - y_{ki}<0} \kappa_k  \frac{c(s)^{y_{k}}}{c_i(s)}(y_{ki}' - y_{ki}) \tag{since $y_{ki} \ge 1$}\\
   & \geq c_i(s) \sum_{k:y_{ki}^\prime - y_{ki}<0}  \kappa_k m^{\|y_k\|_1-1}(y_{ki}' - y_{ki}), \tag{since $\|c(s)\|_1\leq m$ for any $s\in [0,t]$}
\end{align*}
which implies $c_i(t) >0$ for any $t< T^*$.

We will now show that the dynamics of $c(t)$ are linear for $t < T^*$.  Denote the linkage classes of $\C$ by $\mathcal{L}_1, \mathcal{L}_2, \dots, \mathcal{L}_n$.  We have 
\begin{align*}
\frac{d}{dt}{c}(t) &= \sum_{k=1}^K  \kappa_k  c(t)^{y_{k}}(y_k' - y_k) = \sum_{z\in \mathcal{C}} z \left( \sum_{k: y_k^\prime=z} \kappa_k  c(t)^{y_{k}} - \sum_{k: y_k=z} \kappa_k  c(t)^{y_{k}} \right) \\
& = \sum_{\ell} \sum_{z\in \mathcal{L}_\ell} z \left( \sum_{k: y_k^\prime=z} \kappa_k  c(t)^{y_{k}} - \sum_{k: y_k=z} \kappa_k  c(t)^{y_{k}} \right). 
\end{align*}

Our goal is to show that for any linkage class $\mathcal{L}_\ell$, the summation 
\begin{align}\label{eq:linkageclass}
\sum_{z\in \mathcal{L}_\ell} z \left( \sum_{k: y_k^\prime=z} \kappa_k  c(t)^{y_{k}} - \sum_{k: y_k=z} \kappa_k  c(t)^{y_{k}} \right)
\end{align}
only contributes linear terms to the dynamics of the process,  and hence the overall dynamics of the deterministic model  \eqref{eq:mass-action} is linear. 

We now restrict ourselves to the summation \eqref{eq:linkageclass}. There are three cases that we consider. 

\begin{description}
\item[Case 1.] Suppose the linkage class $\mathcal{L}_\ell$ contains only higher order complexes. Then every  term in the summation \eqref{eq:linkageclass} is zero by the DR condition \eqref{eq:complex_balance}.  Thus,
\begin{align*}
 \sum_{z\in \mathcal{L}_\ell} z \left( \sum_{k: y_k^\prime=z} \kappa_k  c(t)^{y_{k}} - \sum_{k: y_k=z} \kappa_k  c(t)^{y_{k}} \right)  = \sum_{z\in \mathcal{L}_\ell} 0  =  0. 
\end{align*}

\item[Case 2.] Suppose the linkage class $\mathcal{L}_\ell$ contains only \edited{zeroth-order} and first-order complexes, then 
\begin{align*}
 \sum_{z\in \mathcal{L}_\ell} z \left( \sum_{k: y_k^\prime=z} \kappa_k  c(t)^{y_{k}} - \sum_{k: y_k=z} \kappa_k  c(t)^{y_{k}} \right) 
\end{align*}
only contributes linearly. 

\item[Case 3.] We now suppose the linkage class $\mathcal{L}_\ell$ contains both higher-order and \edited{lower-order} complexes. Suppose $z_1,z_2,\dots, z_m$ are the higher order complexes and  that $z_{m+1},\dots, z_{|\mathcal{L}_\ell|}$ are \edited{zeroth-order} and first-order complexes. We will follow the idea in Example \ref{example5} by moving all the nonlinear monomials to one side of the equation, and solving for them in terms of the linear terms.  To do so, we change notation slightly by explicitly enumerating the reactions and their rate constants by the reactions themselves.  That is, for $y \to y'\in \mathcal{R}$, we write $\kappa_{y \to y'}$.  We stress that this change is isolated to this portion of the proof.

After making this change in notation, we can write the DR condition for complex  $z_i$, $i=1,\ldots, m$, as
\begin{align*}
\sum_{j =1}^{|\mathcal{L}_\ell|} \kappa_{z_i \to z_j} c(t)^{z_i}  - \sum_{j=1}^m  \kappa_{z_j \to z_i} c(t)^{z_j}  
=\sum_{j=m+1}^{|\mathcal{L}_\ell|} \kappa_{z_j  \to z_i} c(t)^{z_j},
\end{align*}
where we take $\kappa_{y \to y'} = 0$ if $y\to y' \notin \mathcal{R}$.

We have $m$ such conditions, and so we can rewrite the DR condition \eqref{eq:complex_balance} as a vector equation $A \tilde x = b$, where
\begin{enumerate}[(1)]
\item $\tilde x$ is an $m\times 1$ column vector whose $i^{th}$ component is given by $\tilde x_i = c(t)^{z_i}$ for $i=1,\ldots, m$. That is, the vector $\tilde x$ contains all the higher order monomials in the linkage class $\mathcal{L}_\ell$. 
\item $b$ is an $m \times 1$ column vector whose $i^{th}$ component is given by
\begin{align*}
b_i =\sum_{j=m+1}^{|\mathcal{L}_\ell|} \kappa_{z_j  \to z_i} c(t)^{z_j},
\end{align*}
which are all linear.

\item $A$ is an $m\times m$ matrix whose entries are defined as
\begin{align}\label{eq:matrixA}
A_{ii} = \sum_{j =1}^{|\mathcal{L}_\ell|} \kappa_{z_i \to z_j}  \geq 0 \qquad  \text{ and for $j\neq i$,}\qquad  A_{ij} = -  \kappa_{z_j \to z_i} \le 0.
\end{align}
\end{enumerate}

Hence $A_{ij} < 0$ if and only if $z_j\rightarrow z_i \in \mathcal{R}$. Notice that if we can show $A$ is invertible, then we can write $\tilde x= A^{-1}b$.  In this situation, all the higher order monomials can be expressed using first order monomials and hence \eqref{eq:linkageclass} 
can be written as linear combinations of first-order monomials and the dynamics will be linear. 
\vspace{.1in}

It will be more convenient to work with the transpose matrix, $A^T$. The row sums of $A^{T}$ corresponds to column sums of $A$, hence for the $i^{th}$ row 
\begin{align}\label{eq:asdasdsadasdasd}
 \sum_{j=1}^m  (A^T)_{ij} &= A_{ii} + \sum_{j\neq i}^m A_{ji} =  \sum_{j =1}^{|\mathcal{L}_\ell|} \kappa_{z_i \to z_j} - \sum_{j=1}^m \kappa_{z_i \to z_j}  =    \sum_{j = m+1}^{|\mathcal{L}_\ell|} \kappa_{z_i \to z_j} \geq 0,
\end{align}
which implies $A^T$ is weakly diagonally dominant matrix.  Moreover, row $i$ is not SDD if and only if $ \kappa_{z_i \to z_j} = 0$ for $j = m+1,\ldots,|\mathcal{L}_\ell |$, i.e., there is no reaction from $z_i$ to a lower order complex. To finish our proof that $A$ is invertible, we will prove the following claim. 

\vspace{.1in}

\textbf{Claim: } If $c(t)$ satisfies the DR condition of Definition \ref{def:DR}, then the path condition in Lemma \ref{lemma:WCDD} holds for $A^T$.

\begin{proof}[Proof of the claim.] 
First, we consider the associated directed graph of the matrix $A^T$. Notice that by \eqref{eq:matrixA},  $(A^T)_{ij} \neq 0$ if and only if $\kappa_{z_i\to z_j } > 0$, i.e. , $z_i \to z_j \in \Re$. Hence the associated directed graph is equivalent to our reaction graph, where row $i$ corresponds to complex $z_i$ in the reaction graph. Then by \eqref{eq:asdasdsadasdasd}, row $i$ is not SDD if and only if $ \kappa_{z_i \to z_j} = 0$ for $j = m+1,\ldots,|\mathcal{L}_\ell |$, i.e., there is no reaction from $z_i$ to a lower order complex. 

Suppose, in order to find a contradiction, that the path condition does not hold for $A^T$.  Specifically, we assume there  exists a row $i_1$ which can not reach a row that is SDD in the associated directed graph. Then, consider the following set of complexes 
\[
  \tilde{C} = \{ z\in \mathcal{L}_\ell : \text{ there is a path from $z_{i_1}$ to $z$} \} \subset \mathcal{L}_\ell.
\]

Then $z\notin \tilde C$ for any $\|z\|_1\leq 1$, since otherwise, there exists a reaction from higher order complex to lower order complex along the path from $z_{i_1}$ to $z$, which contradicts with the fact that all rows are not SDD. Consequently, $\|z \|_1\geq 2$ for any $z\in \tilde C$. Therefore, by the DR condition for all complexes $z\in \tilde{C}$, we have
\begin{align*}
  \sum_{z\in \tilde{C} }  \sum_{k: y_k^\prime=z} \kappa_k  c(t)^{y_{k}} = \sum _{z\in \tilde{C} } \sum_{k: y_k=z} \kappa_k  c(t)^{y_{k}} ,
\end{align*}
which immediately leads to the equation,
\begin{align}\label{eq:balancing}
 \sum_{k: y_k^\prime \in \tilde{C} } \kappa_k  c(t)^{y_{k}} = \sum_{k: y_k \in \tilde{C} } \kappa_k  c(t)^{y_{k}}.
\end{align}
If $y_k\in \tilde C$, then $y_k^\prime \in \tilde C$ since there exists a path connecting $z_{i_1}$ and $y_k^\prime$ via $y_k$. That is, 
\[
	\{k: y_k^\prime \in \tilde{C} \} \supseteq \{k: y_k \in \tilde{C} \}.
\]
 Given that they have the same summands in \eqref{eq:balancing} and $c(t) > 0$ for $t< T^*$, the index sets are equal $\{ k: y_k^\prime \in \tilde{C}  \} = \{ k: y_k \in \tilde{C}  \}$. However this would imply $\tilde{C}$ is a linkage class by itself, as for any complex $z \in \tilde C$ and $z^\prime \notin \tilde C$, $z\to z^\prime \notin \Re$ and $z^\prime\to  z\notin \Re$. Since $\tilde{C}$ contained strictly inside $\mathcal{L}_\ell$ (first-order complexes are not in $\tilde C$), we get a contradiction. Hence the path condition in Lemma \ref{lemma:WCDD} holds for $A^T$. 
\end{proof}

Given the claim, and by Lemma \ref{lemma:WCDD}, we get $A$ is invertible, and hence \eqref{eq:linkageclass} can be written as linear combinations of first order monomials. 
\end{description}

In conclusion, for each linkages class $\mathcal{L}_\ell$, the summation \eqref{eq:linkageclass} contributes at most linear monomials to the dynamics. Hence the right-hand side of \eqref{eq:mass-action} is linear.

This analysis held under the assumption that $t < T^*$.  However, because we can now conclude that the dynamics are linear for $t < T^*$, we must have that $T^* = \infty$, and the proof is now complete.
\end{proof}

\section{Proofs of Lemmas \ref{Calculation} and \ref{Function_Independent}}
\label{appendix:B}
We restate Lemma \ref{Calculation} for the sake of reference.

\vspace{.1in}

\noindent \textbf{Lemma 3.2} \textit{
Suppose $P_\mu(x,t)$ is given by \eqref{poissonAT} with $c(t) \in \R_{>0}^d$ for all $t \ge 0$.   Then $P_\mu(x,t)$ is the solution to the Kolmogorov forward equation  \eqref{eq:CME} if and only if $c(t)$ satisfies the deterministic equation \eqref{eq:mass-action} and 
\begin{align}
 \sum_{k} \kappa_k c(t)^{y_k}  \left[  g_{x,c(t)}(y_k^\prime) -g_{x,c(t)}(y_k)  \right] = 0 \tag{20}
\end{align}
where for each $x \in \Z^d_{\ge 0}$ and $c \in \R^d_{ > 0}$,
\begin{align}
g_{x,c}(y_k) =  \sum_{j=1}^d \left( \frac{x_j}{c_j}-1\right) y_{kj} -  \frac{x !}{(x -y_{k})!} c^{-y_k} +1. \tag{21}
\end{align}
Moreover, if $\|y_k\|_1 \leq 1$, then $g_{x,c} (y_k) = 0$. 
}

\begin{proof}
We will first assume that $P_\mu(x,t)$ is as in \eqref{poissonAT} and that it is the solution to the Kolmogorov forward equation\eqref{eq:CME}. Our goal is to show that \eqref{eq:eq} holds.   

By Proposition \ref{main_theorem}, $c(t)$ satisfies \eqref{eq:mass-action}.  In particular, it is differentiable.   Because  $P_\mu(x,t)$ is as in \eqref{poissonAT}, the left-hand side of \eqref{eq:CME}  satisfies
\begin{align}
\frac{d}{dt} P_\mu(x,t) &= \frac{d}{dt} \left(  \prod_{i=1}^d  e^{-c_i(t)}\frac{c_i(t)^{x_i}}{x_i ! } \right) \notag\\
& = \sum_{j=1}^d  \prod_{i\neq j} e^{-c_i(t)}\frac{c_i(t)^{x_i}}{x_i ! } \left(-c_j^\prime(t)e^{-c_j(t)}\frac{c_j(t)^{x_j}}{x_j ! }  + x_je^{-c_j(t)}\frac{c_j(t)^{x_j-1}}{x_j ! } c_j^\prime(t) \right) \notag\\
& = \prod_{i=1}^d e^{-c_i(t)}\frac{c_i(t)^{x_i}}{x_i ! } \sum_{j=1}^d \left( -c_j^\prime (t) + x_j \frac{c_j^\prime (t) }{c_j(t) } \right)\notag \\
& =  e^{-c(t)}\frac{c(t)^{x}}{x ! }  \sum_{j=1}^d c_j^\prime (t) \left( \frac{x_j}{c_j(t) }-1\right) \notag\\
& =  e^{-c(t)}\frac{c(t)^{x}}{x ! }  \sum_{j=1}^d \sum_{k=1}^K \kappa_k c(t)^{y_k} (y_{kj}^{\prime} - y_{kj}) \left( \frac{x_j}{c_j(t) }-1\right) \notag\\
& = \left( e^{-c(t)}\frac{c(t)^{x}}{x ! } \right) \sum_{k=1}^K \kappa_k c(t)^{y_k}  \sum_{j=1}^d \left( \frac{x_j}{c_j(t) }-1\right)  (y_{kj}^{\prime} - y_{kj}).\label{eq:567876545678}
\end{align}
The right hand side of \eqref{eq:CME} is
\begin{align}
 \sum_{k=1}^K &\lambda_{k}(x-\zeta_k) P_\mu(x-\zeta_k ,t)  - \sum_{k=1}^K \lambda_k(x) P_\mu(x,t)\notag \\
=  & \sum_{k=1}^K   \kappa_k \left( \frac{(x-\zeta_k)!}{(x-\zeta_k-y_k)!}  e^{-c(t)}\frac{c(t)^{x-\zeta_{k}}}{(x-\zeta_{k} )! }\right)   - \sum_{k=1}^K  \kappa_k \left(  \frac{x !}{(x -y_{k})!}  e^{-c(t)}\frac{c(t)^{x}}{x! } \right)\notag\\ 
= & \left( e^{-c(t)}\frac{c(t)^{x}}{x ! } \right)  \sum_{k=1}^K  \kappa_k     \left( \frac{  x!}{(x -\zeta_{k} - y_{k})!} c(t)^{-\zeta_{k}}  -  \frac{x !}{(x -y_{k})!}  \right) \notag\\
= & \left( e^{-c(t)}\frac{c(t)^{x}}{x ! } \right)   \sum_{k=1}^K  \kappa_k  c(t)^{y_k}  \left( \frac{  x!}{(x -y_{k}^\prime)!} c(t)^{-y_k^\prime} -  \frac{x !}{(x -y_{k})!} c(t)^{-y_k} \right).\label{eq:7658924028374}
\end{align}
Since $P_\mu(x,t)$ is the solution to \eqref{eq:CME}, we must have that \eqref{eq:567876545678} and \eqref{eq:7658924028374} are equal.  That is,
\begin{align}
&\sum_{k=1}^K \kappa_k c(t)^{y_k}  \left(  \sum_{j=1}^d \bigg[\left( \frac{x_j}{c_j(t) }-1\right) (y_{kj}^{\prime} - y_{kj}) -  \left( \frac{  x!}{(x -y_{k}^\prime)!} c(t)^{-y_k^\prime} -  \frac{x !}{(x -y_{k})!} c(t)^{-y_k} \right) \bigg] \right)= 0.
 \label{eq:jkdfakjdkja}
\end{align}

Define the following function 
\begin{align*}
f_{x,c}(y_k) =  \sum_{j=1}^d \left( \frac{x_j}{c_j }-1\right) y_{kj} -  \frac{x !}{(x -y_{k})!} c^{-y_k} 
\end{align*}
and let $\displaystyle g_{x,c}(y_k) = f_{x,c} (y_k) +1$.  Then we can rewrite  equation \eqref{eq:jkdfakjdkja} above as 
\begin{align*}
 \sum_{k=1}^K \kappa_k c(t)^{y_k}  \left[  g_{x,c(t)}(y_k^\prime) - g_{x,c(t)}(y_k)  \right] = 0,
\end{align*}
which shows \eqref{eq:eq} holds.   

 To show the other direction, suppose $c(t)$ is the solution to the deterministic equation \eqref{eq:mass-action} and that  \eqref{eq:eq} is satisfied.  We must show that  $P_\mu(x,t)$ as in \eqref{7697679} is the  solution to the Kolmogorov forward equation \eqref{eq:CME}.   However, this follows by reversing the steps above.

All that remains is to demonstrate that if $\|y_k\|_1 \le 1$, then $g_{x,c}(y_k) = 0$.
There are only two cases that need consideration.

\vspace{.1in}

\noindent \textit{Case 1.} If $y_k =  \vec{0} $, then 
\begin{align*}
g_{x,c}(y_k) =  \sum_{j=1}^d \left( \frac{x_j}{c_j }-1\right) y_{kj} - \frac{x !}{(x -y_{k})!} c^{-y_k} +1 = 0-1 +1= 0.
\end{align*}

\vspace{.1in}

\noindent \textit{Case 2.} If $y_k =  e_\ell$, the vector whose $\ell^{th}$ entry is 1 and all other entries are zero, then 
\begin{align*}
g_{x,c}(y_k) =  \sum_{j=1}^d \left( \frac{x_j}{c_j }-1\right) y_{kj} -  \frac{x !}{(x -y_{k})!} c^{-y_k} +1 = \frac{x_\ell}{c_\ell} - 1 - \frac{x_\ell}{c_\ell}  +1= 0.
\end{align*}
Hence, the proof is complete.
\end{proof}

We restate Lemma \ref{Function_Independent} for the sake of reference.\\

\noindent \textbf{Lemma 3.3} \textit{
Let $\{z_1,z_2, ...., z_m\} \subset \C$ be the collection of complexes that are at least binary (i.e.~$\|z_i\|_1 \ge 2$). Fix a value $c \in \R^d_{>0}$.  For each $i \in \{1,\dots, m\}$ let $f_i: \Z_{\ge 0}^d \rightarrow \R$ be defined as
\[
	f_i(x) = g_{x,c}(z_i),
\]
where the functions $g_{x,c}$ are defined in the proof of Lemma \ref{Calculation}.
Then $\{ f_i\}_{i=1}^m$ are linear independent. 
}

\vspace{.1in}

The main idea of the proof rests on noticing that this collection of functions consists of polynomials of different leading orders. An example will be helpful to illustrate.  Let us turn to the binary case with two species, and denote $\mathcal{C} = \{2e_1, 2e_2, e_1+e_2\} $. Then the relevant functions are
\begin{align*}
f_1(x) & = 2 \left( \frac{x_1}{c_1 }-1\right)  -  \frac{x_1(x_1-1)}{c_1^2 }  +1 = - \frac{x_1^2}{c_1^2} + \left( 2+\frac{1}{c_1} \right) \frac{x_1}{c_1} -1\\
f_2(x) & = 2 \left( \frac{x_2}{c_2 }-1\right)  -  \frac{x_2(x_2-1)}{c_2^2 }  +1 = - \frac{x_2^2}{c_2^2} + \left( 2+\frac{1}{c_2} \right) \frac{x_2}{c_2} -1\\
f_3(x) & =  \left( \frac{x_1}{c_1 }-1\right)+\left( \frac{x_2}{c_2}-1\right)  -  \frac{x_1 x_2}{c_1c_2 }  +1 = -\frac{x_1 x_2}{c_1c_2 } + \frac{x_1}{c_1} + \frac{x_2}{c_2} -1.
\end{align*}
To see why they are linearly independent, let $\alpha_i$ be such that $\alpha_1 f_1(x) + \alpha_2 f_2(x) + \alpha_3 f_3(x) = 0$ for all $x$. Since the leading powers of the monomials are different, we therefore conclude that we must have $\alpha_1 = \alpha_2 = \alpha_3 = 0$.

\begin{proof}[Proof of Lemma \ref{Function_Independent}]
Suppose there exists $\alpha_i$ for $i=1,2,..., m$ such that 
\begin{align*}
\alpha_1 f_1(x) +\cdots + \alpha_m f_m(x) = 0,
\end{align*}
for all $x \in \Z^d_{\ge 0}$.

Let $\displaystyle s=\max_{i=1,2,...,m} \|z_i\|_1 $ and denote $\mathcal{\tilde{C}} = \{ z_i:  \|z_i \|_1 = s\}$. Notice that for any function $f_i$ where $z_i \in \tilde{\mathcal{C}}$, $f_i(x)$ is a polynomial in $x$ and the leading term of the polynomial is $\frac{1}{c^{z_i}}x^{z_i}$. Notice that  for $i\neq j$, we have $z_i \neq z_j$ and hence $x^{z_i} \neq x^{z_j}$.  We may therefore conclude that $\alpha_i  = 0$ for any $z_i \in \tilde \C$.

The proof is then concluded by noting that the above procedure can be performed iteratively as you decrease the 1-norm of the complexes.
%
%
\end{proof}

 \bibliographystyle{plain}
\bibliography{poisson_time}

\providecommand{\noopsort}[1]{}\providecommand{\singleletter}[1]{#1}%
\begin{thebibliography}{10}

\bibitem{ACKK2018}
David~F. Anderson, Daniele Cappelletti, Masanori Koyama, and Thomas~G. Kurtz.
\newblock Non-explosivity of stochastically modeled reaction networks that are
  complex balanced.
\newblock {\em Bull. Math. Biol.}, 80(10):2561--2579, 2018.

\bibitem{AC2016}
David~F. Anderson and Simon~L. Cotter.
\newblock Product-form stationary distributions for deficiency zero networks
  with non-mass action kinetics.
\newblock {\em Bull. Math. Bio.}, 78:2390--2407, 2016.

\bibitem{AndProdForm}
David~F. Anderson, Gheorghe Craciun, and Thomas~G. Kurtz.
\newblock Product-form stationary distributions for deficiency zero chemical
  reaction networks.
\newblock {\em Bull. Math. Biol.}, 72(8):1947--1970, 2010.

\bibitem{AndKurtz2011}
David~F. Anderson and Thomas~G. Kurtz.
\newblock Continuous time markov chain models for chemical reaction networks.
\newblock In H.~Koeppl et~al., editor, {\em Design and Analysis of Biomolecular
  Circuits: Engineering Approaches to Systems and Synthetic Biology}, pages
  3--42. Springer, 2011.

\bibitem{AK2015}
David~F. Anderson and Thomas~G. Kurtz.
\newblock {\em Stochastic analysis of biochemical systems}, volume 1.2 of {\em
  Stochastics in Biological Systems}.
\newblock Springer International Publishing, Switzerland, 1 edition, 2015.

\bibitem{anderson2015stochastic}
David~F Anderson and Thomas~G Kurtz.
\newblock {\em Stochastic analysis of biochemical systems}, volume~1.
\newblock Springer, 2015.

\bibitem{cao2018linear}
Zhixing Cao and Ramon Grima.
\newblock Linear mapping approximation of gene regulatory networks with
  stochastic dynamics.
\newblock {\em Nature communications}, 9(1):3305, 2018.

\bibitem{CW2016}
Daniele Cappelletti and Carsten Wiuf.
\newblock Product-form poisson-like distributions and complex balanced reaction
  systems.
\newblock {\em SIAM J. Appl. Math.}, 76(1):411--432, 2016.

\bibitem{Kurtz86}
Stewart~N. Ethier and Thomas~G. Kurtz.
\newblock {\em Markov Processes: Characterization and Convergence}.
\newblock John Wiley \& Sons, New York, 1986.

\bibitem{F1}
Martin Feinberg.
\newblock Complex balancing in general kinetic systems.
\newblock {\em Arch. Ration. Mech. Anal.}, 49:187--194, 1972.

\bibitem{gardiner1985stochastic}
Crispin Gardiner.
\newblock Stochastic methods.
\newblock {\em Springer Series in Synergetics (Springer-Verlag, Berlin, 2009)},
  1985.

\bibitem{gardiner1977poisson}
Crispin Gardiner and S~Chaturvedi.
\newblock The poisson representation. {I}. {A} new technique for chemical
  master equations.
\newblock {\em Journal of Statistical Physics}, 17(6):429--468, 1977.

\bibitem{gillespie1992rigorous}
Daniel~T Gillespie.
\newblock A rigorous derivation of the chemical master equation.
\newblock {\em Physica A: Statistical Mechanics and its Applications},
  188(1-3):404--425, 1992.

\bibitem{gillespie2001approximate}
Daniel~T Gillespie.
\newblock Approximate accelerated stochastic simulation of chemically reacting
  systems.
\newblock {\em The Journal of Chemical Physics}, 115(4):1716--1733, 2001.

\bibitem{H}
Fritz Horn.
\newblock Necessary and sufficient conditions for complex balancing in chemical
  kinetics.
\newblock {\em Arch. Ration. Mech. Anal.}, 49:172--186, 1972.

\bibitem{horn1972necessary}
Fritz Horn.
\newblock Necessary and sufficient conditions for complex balancing in chemical
  kinetics.
\newblock {\em Archive for Rational Mechanics and Analysis}, 49(3):172--186,
  1972.

\bibitem{H-J1}
Fritz Horn and Roy Jackson.
\newblock General mass action kinetics.
\newblock {\em Arch. Ration. Mech. Anal.}, 47:187--194, 1972.

\bibitem{horn2012matrix}
Roger~A Horn and Charles~R Johnson.
\newblock {\em Matrix analysis}.
\newblock Cambridge university press, 2012.

\bibitem{jahnke2007solving}
Tobias Jahnke and Wilhelm Huisinga.
\newblock Solving the chemical master equation for monomolecular reaction
  systems analytically.
\newblock {\em Journal of mathematical biology}, 54(1):1--26, 2007.

\bibitem{munsky2018distribution}
Brian Munsky, Guoliang Li, Zachary~R Fox, Douglas~P Shepherd, and Gregor
  Neuert.
\newblock Distribution shapes govern the discovery of predictive models for
  gene regulation.
\newblock {\em Proceedings of the National Academy of Sciences},
  115(29):7533--7538, 2018.

\bibitem{neuert2013systematic}
Gregor Neuert, Brian Munsky, Rui~Zhen Tan, Leonid Teytelman, Mustafa Khammash,
  and Alexander van Oudenaarden.
\newblock Systematic identification of signal-activated stochastic gene
  regulation.
\newblock {\em Science}, 339(6119):584--587, 2013.

\bibitem{peccoud1995markovian}
Jean Peccoud and Bernard Ycart.
\newblock Markovian modeling of gene-product synthesis.
\newblock {\em Theoretical population biology}, 48(2):222--234, 1995.

\bibitem{ramos2011exact}
Alexandre~Ferreira Ramos, Guilherme~C.P. Innocentini, and Jos{\'e}
  Eduardo~Martinho Hornos.
\newblock Exact time-dependent solutions for a self-regulating gene.
\newblock {\em Physical Review E}, 83(6):062902, 2011.

\bibitem{schnoerr2017approximation}
David Schnoerr, Guido Sanguinetti, and Ramon Grima.
\newblock Approximation and inference methods for stochastic biochemical
  kinetics---a tutorial review.
\newblock {\em Journal of Physics A: Mathematical and Theoretical},
  50(9):093001, 2017.

\bibitem{shahrezaei2008analytical}
Vahid Shahrezaei and Peter~S Swain.
\newblock Analytical distributions for stochastic gene expression.
\newblock {\em Proceedings of the National Academy of Sciences},
  105(45):17256--17261, 2008.

\bibitem{shivakumar1974sufficient}
PN~Shivakumar and Kim~Ho Chew.
\newblock A sufficient condition for nonvanishing of determinants.
\newblock {\em Proceedings of the American mathematical society}, pages 63--66,
  1974.

\bibitem{smadbeck2012efficient}
P~Smadbeck and YN~Kaznessis.
\newblock Efficient moment matrix generation for arbitrary chemical networks.
\newblock {\em Chemical engineering science}, 84:612--618, 2012.

\bibitem{wilkinson2006stochastic}
Darren~J Wilkinson.
\newblock {\em Stochastic modelling for systems biology}.
\newblock Chapman and Hall/CRC, 2006.

\bibitem{zechner2012moment}
Christoph Zechner, Jakob Ruess, Peter Krenn, Serge Pelet, Matthias Peter, John
  Lygeros, and Heinz Koeppl.
\newblock Moment-based inference predicts bimodality in transient gene
  expression.
\newblock {\em Proceedings of the National Academy of Sciences},
  109(21):8340--8345, 2012.

\end{thebibliography}

\end{document}